\documentclass{article}
\usepackage{amsmath,amsthm,amssymb,cases,amscd, bm}
\usepackage{indentfirst}
\usepackage{ascmac}
\usepackage{yhmath}
\usepackage{mathrsfs}
\usepackage{textcomp}
\usepackage{braket}
\usepackage{mathtools}
\usepackage{enumerate}

\usepackage[top=23truemm,bottom=25truemm,left=20truemm,right=20truemm]{geometry}
\usepackage{bookmark}

\newtheorem{lemma}{Lemma}[section]
\newtheorem{theorem}[lemma]{Theorem}
\newtheorem{proposition}[lemma]{Proposition}
\newtheorem{corollary}[lemma]{Corollary}

\theoremstyle{definition}
\newtheorem*{remark}{Remark}

\newcommand{\C}{\mathbb{C}}

\newcommand{\Z}{\mathbb{Z}}

\newcommand{\GL}{\operatorname{GL}}
\newcommand{\SL}{\operatorname{SL}}
\newcommand{\PGL}{\operatorname{PGL}}
\newcommand{\PSL}{\operatorname{PSL}}
\newcommand{\SO}{\operatorname{SO}}
\newcommand{\Or}{\operatorname{O}}

\newcommand{\GSp}{\operatorname{GSp}}
\newcommand{\PGSp}{\operatorname{PGSp}}

\newcommand{\Hom}{\operatorname{Hom}}

\newcommand{\Irr}{\operatorname{Irr}}
\newcommand{\Res}{\operatorname{Res}}
\newcommand{\Ind}{\operatorname{Ind}}
\newcommand{\cInd}{\operatorname{c-Ind}}
\newcommand{\old}{\mathrm{old}}
\newcommand{\new}{\mathrm{new}}

\title{Conductors and local newforms for the metaplectic group of rank 1}
\author{Hiroshi Ishimoto\thanks{Osaka Metropolitan University}}
\date{[\today]}

\begin{document}

\maketitle

\begin{abstract}
    In an earlier paper of W. Casselman, the theory of local newforms and conductors was initiated.
    Later, Roberts and Schmidt studied local newforms for the metaplectic group of rank 1.
    In this paper we define and calculate conductors of irreducible genuine representations of the metaplectic group of rank 1 over non-archimedean local field of characteristic zero and of odd residual characteristic.
    Moreover, we shall give an explicit formulae for dimensions of spaces of local newforms, and show a compatibility with the local theta correspondence.
\end{abstract}

\tableofcontents

\section{Introduction}
\subsection*{Background \& Main results}
In 1970, Atkin--Lehner \cite{al} introduced the notion of newforms for elliptic modular forms.
After that, Casselman \cite{cas} studied the Atkin--Lehner theory in the framework of local and global representations of $\GL_2$.
The theory of local newforms for $\GL_2$ given by Casselman in loc. cit. was generalized to the case for $\GL_N$ by Jacquet--Piatetski-Shapiro--Shalika \cite{jpss} and Atobe--Kondo--Yasuda \cite{aky}.
Moreover, the theory of local newforms has also been developed for several other groups.
For example, $\SL_2$ was treated by Lansky--Raghuram \cite{lr}, $\PGSp_4$ and $\widetilde{\SL_2}$ by Roberts--Schmidt \cite{rs-GSp4,rs-Mp2}, and $\SO_{2n+1}$ by Tsai \cite{tsai} and Cheng \cite{cheng}.

In \cite{rs-Mp2}, Roberts--Schmidt gave the number of local newforms for $\widetilde{\SL_2}$.
They defined the spaces of local newforms, and computed their dimensions without dealing with the conductors.
Let us review their theory briefly.
Let $F$ be a $p$-adic field whose residual characteristic is odd, $\mathcal{O}_F$ the ring of integers, and $\varpi_F$ a prime element.
The group $\widetilde{\SL_2}(F)$ is a unique nonlinear two-fold cover of $\SL_2(F)$ with the exact sequence
\begin{equation*}
    1 \longrightarrow \{\pm1\} \longrightarrow \widetilde{\SL_2}(F) \longrightarrow \SL_2(F) \longrightarrow 1.
\end{equation*}
We identify $\widetilde{\SL_2}(F)$ with $\SL_2(F)\times\{\pm1\}$ as sets.
Let $K^0_0=\SL_2(\mathcal{O}_F)$ and $K^1_0=\alpha^{-1} K_0 \alpha$, where $\alpha=\operatorname{diag}(\varpi_F,1)$.
They are, up to conjugacy, the two maximal open compact subgroups of $\SL_2(F)$.
For every integer $m\geq1$, let $K^0_m$ be the subgroup of $K^0_0$ consisting of matrices whose $(2,1)$-components belong to $\varpi_F^m\mathcal{O}_F$, and put $K^1_m=\alpha^{-1} K^0_m\alpha$.
For each $\varepsilon \in \{0,1\}$, fix a nontrivial additive character $\psi^\varepsilon$ of $F$ of conductor $-\varepsilon$.
The exact sequence above is known to split over $K^\varepsilon_0$, and hence $K^\varepsilon_0$ can be regarded as a subgroup of $\widetilde{\SL_2(F)}$.
Here, splitting is uniquely determined by $\psi^\varepsilon$.
Let $(\pi,V)$ be an irreducible admissible (infinite-dimensional) genuine representation of $\widetilde{\SL_2}(F)$, where we say that a representation $\pi$ of $\widetilde{\SL_2}(F)$ is genuine if it does not factor through $\SL_2(F)$.
For any character $\eta$ of $\mathcal{O}_F^\times$, put
\begin{align*}
    \pi^{K^\varepsilon_m}_\eta = \Set{v\in V | \pi\left( \left( \begin{array}{cc}a&b\\ c&d \end{array} \right) \right)v=\eta(d)v, \text{ for all } \left( \begin{array}{cc}a&b\\ c&d \end{array} \right)\in K^\varepsilon_m}.
\end{align*}
Note that there is a natural inclusion from $\pi^{K^\varepsilon_{m-1}}_\eta$ into $\pi^{K^\varepsilon_m}_\eta$.
Roberts--Schmidt \cite{rs-Mp2} introduced a linear operator $\alpha_2$ from $V$ to itself, which is an analog of the Atkin--Lehner involution and sends $\pi^{K^\varepsilon_{m-2}}_\eta$ to $\pi^{K^\varepsilon_m}_\eta$.
They defined the subspace $\pi^{K^\varepsilon_m, \old}_\eta$ of $\pi^{K^\varepsilon_m}_\eta$ spanned by these two images, and called vectors in the subspace oldforms.
Put
\begin{equation*}
    \pi^{K^\varepsilon_m, \new}_\eta = \pi^{K^\varepsilon_m}_\eta / \pi^{K^\varepsilon_m, \old}_\eta.
\end{equation*}
Then the sum
\begin{equation*}
    \sum_{m\geq 0} \dim_\C \pi^{K^\varepsilon_m, \new}_\eta
\end{equation*}
is finite and can be given explicitly.
Our aim in this paper is to describe the newforms more precisely.

We define the $\eta$-conductor
\begin{equation*}
    c^\varepsilon_\eta(\pi)=\min\left( 0\leq m \leq +\infty \mid \pi^{K^\varepsilon_m}_\eta\neq \{0\} \right).
\end{equation*}
We also put $c^\varepsilon_{\min}(\pi)=\min(c^\varepsilon_\eta(\pi) \mid \eta)$, where $\eta$ runs over characters of $\mathcal{O}_F^\times$.
Our first main result (Theorems \ref{thmps}, \ref{thmevenweil}, \ref{thmsteinberg}, and \ref{thmsc}) is:
\begin{theorem}\label{thm:main1}
    Assume that the conductor of $\eta$ is less than or equal to $c^\varepsilon_{\min}(\pi)/2$ if $\pi$ is supercuspidal.
    We have explicit formulae for the dimensions of the spaces $\pi^{K^\varepsilon_m}_\eta$, $\pi^{K^\varepsilon_m, \new}_\eta$.
    In particular, $c^\varepsilon_{\min}(\pi)$ and $c^\varepsilon_1(\pi)$ are also given explicitly.
\end{theorem}
From Theorem \ref{thm:main1}, we can see that the newforms appear early (Corollaries \ref{corps}, \ref{corevenweil}, \ref{corsteinberg}, and \ref{corsc}).
\begin{corollary}
    Assume that $c^\varepsilon_\eta(\pi)$ is not $+\infty$.
    Under the assumption of Theorem \ref{thm:main1}, there is an integer $a_{\pi,\varepsilon,\eta}\in \{0, 1, 2, 3\}$ such that $\pi^{K^\varepsilon_m, \new}_\eta$ is nonzero if $c^\varepsilon_\eta(\pi) \leq m \leq c^\varepsilon_\eta(\pi) + a_{\pi,\varepsilon,\eta}$ and zero otherwise, and the integer $a_{\pi,\varepsilon,\eta}$ is given explicitly.
\end{corollary}
To prove Theorem \ref{thm:main1}, we reduce it to the case of $\SL_2$, the result of Lansky--Raghuram \cite{lr}.
The irreducible principal series representations and the supercuspidal ones of $\widetilde{\SL_2}(F)$ can be constructed in parallel with those of $\SL_2(F)$.
Thanks to this fact, the theorem can be proved in a similar way to Propositions in \cite[\S3.2]{lr} if $\pi$ is an irreducible principal series representation, and we can construct an irreducible representation $\pi'$ of $\SL_2(F)$ and a linear isomorphism between $\pi^{K^\varepsilon_m}_\eta$ and $(\pi')^{K^\varepsilon_m}_\eta$ for any $m\geq0$, if $\pi$ is supercuspidal.
For a nontrivial subquotient of a principal series representation, i.e., an even Weil representation or a Steinberg representation, a different treatment is required.
In the case that $\pi$ is an even Weil representation, which has the Schr\"{o}dinger model, we determine the spaces $\pi^{K^\varepsilon_m}_\eta$ in the model.
Then it remains to consider the case that $\pi$ is a Steinberg representation.
It is a unique irreducible nontrivial subrepresentation of a principal series representation with the quotient isomorphic to an even Weil representation, and this short exact sequence gives the dimensions of $\pi^{K^\varepsilon_m}_\eta$.
We have an additional property on newforms.
In \cite{cas}, Casselman proved that any infinite-dimensional representation of $\GL_2(F)$ has only one newform up to a scalar multiple.
This assertion is sometimes referred to as multiplicity one for newforms, and is compatible with the Atkin--Lehner theory on newforms for elliptic modular forms.
In the case of metaplectic group, unlike the case of $\GL_2(F)$, the spaces of newforms $\pi^{K^\varepsilon_m, \new}_\eta$ is not one dimensional in general.
More precisely, its dimension is less than or equal to $2$.
Nevertheless, we will show that appropriate Whittaker functionals do not vanish on these spaces.
Since the image of any nonzero Whittaker functional is one-dimensional, this is our substitute for multiplicity one for newforms.
Our second main result (Corollaries \ref{corwhittakerps}, \ref{corwhittakerevenweil}, \ref{corwhittakersteinberg}, and \ref{cor:whittakersupercuspidal}) is:
\begin{theorem}\label{thm:mainwhittaker_multiplicityone}
    Assume that the conductor of $\eta\mu^{-1}$ is less than or equal to that of $\eta\mu$, if $\pi$ is an irreducible subquotient of a principal series representation induced by a character $\mu$.
    Otherwise, i.e., if $\pi$ is supercuspidal, assume that the conductor of $\eta$ is less than or equal to $c^\varepsilon_{\min}(\pi)/2$.
    Assume moreover that $c^\varepsilon_\eta(\pi)$ is not $+\infty$.
    Let $\Psi$ be a nontrivial additive character of $F$ with conductor $-\varepsilon$ or $-\varepsilon-1$ such that $\pi$ is $\Psi$-generic.
    Then the followings hold.
    \begin{enumerate}[(1)]
        \item If $\Psi$ has conductor $-\varepsilon$, then any $\Psi$-Whittaker functional is nonzero on $\pi^{K^\varepsilon_{c^\varepsilon_\eta(\pi)}}_\eta$.
        \item If $\pi$ is not $\Psi'$-generic for any additive character $\Psi'$ with conductor $-\varepsilon$, then any $\Psi$-Whittaker functional is nonzero on $\pi^{K^\varepsilon_{c^\varepsilon_\eta(\pi)}}_\eta$.
        \item If $\pi$ is $\Psi'$-generic for some additive character $\Psi'$ with conductor $-\varepsilon$ and $\Psi$ has conductor $-\varepsilon-1$, then any $\Psi$-Whittaker functional is nonzero on $\pi^{K^\varepsilon_{c^\varepsilon_\eta(\pi)+a_{\pi,\varepsilon,\eta}}}_\eta$.
    \end{enumerate}
\end{theorem}

In this paper, we have another main result on the conductors and the local theta correspondence.
In a seminal 1973 paper \cite{shi}, Shimura established a lifting from Hecke eigenforms of half-integral weight of level $\Gamma_0(4)$ to those of integral weight of level $\SL(2,\Z)$.
Subsequently in 1980, Kohnen \cite{koh} rewrote it as an isomorphism between the plus space of cusp forms of half-integral weight of level $\Gamma_0(4)$ and the space of those of integral weight of level $\SL(2,\Z)$.
Later, Kohnen \cite{kohnew} and Ueda--Yamana \cite{uy} gave a correspondence between newforms of half-integral weight and those of integral weight.
Since Shimura's lifting is closely related to the theta correspondence, it is natural to expect that the theories of local newforms for $\widetilde{\SL_2}(F)$ and $\GL_2(F)$ are compatible with the local theta correspondence.
The third main result (Theorem \ref{thmtheta}) of this paper is:
\begin{theorem}\label{thm:maintheta}
    Assume that $\pi$ is $\psi^\varepsilon$-generic and $c^\varepsilon_1(\pi)$ is not $+\infty$.
    Then its local theta lift $\theta_{\psi^\varepsilon}(\pi)$ to $\PGL_2(F)$ is nonzero irreducible, and the conductor of its pullback to $\GL_2(F)$ is equal to $c^\varepsilon_1(\pi)$.
\end{theorem}
Theorem \ref{thm:maintheta} is a consequence of Theorem \ref{thm:main1}, since the local theta correspondence between $\widetilde{\SL_2}(F)$ and $\PGL_2(F)$ is given explicitly by Waldspurger \cite{wal1} and Manderscheid \cite{man3}.

\subsection*{Organization}
We shall first give evaluations of two variants of the Gauss sums in \S\ref{secgausssum}.
One variant appear in the calculation of the dimensions of $\pi^{K^\varepsilon_m}_\eta$ when $\pi$ is an even Weil representation, and the other appear in the calculation of Whittaker functionals for principal series representations.
\S\ref{secmetap} is the preliminary section, where the notation for the metaplectic double covering of $\SL_2(F)$ is established and we recall some known results on its representations.
The main theorems are given in \S\S\ref{sec:ps} and \ref{secsc}.
We consider irreducible non-supercuspidal representations in \S\ref{sec:ps}, and supercuspidal ones in \S\ref{secsc}.
In the latter, we review the theories on constructions of supercuspidal representations before the proof of the main theorems for irreducible supercuspidal representations.
In \S\ref{sectheta}, we observe the behaviour of the conductor under the local theta correspondence.

\subsection*{Convention \& Notation}
Let $F$ be a $p$-adic field, i.e., a non-archimedean local field of characteristic 0 with residual characteristic $p$.
In this paper, we assume that $p$ is an odd prime number.
Let $\mathcal{O}$ be the ring of integers of $F$, and $\mathfrak{p}$ the maximal ideal of $\mathcal{O}$.
Fix a prime element $\varpi$ in $\mathcal{O}$.
Put $q=\#(\mathcal{O}/\mathfrak{p})$, and let $|-|=|-|_F$ be the valuation on $F$ normalized so that $|\varpi|=q^{-1}$.
Let $\operatorname{ord} \colon F\to \Z\cup\{\infty\}$ be a function defined by $|x|=q^{-\operatorname{ord}(x)}$ ($x\neq0$) and $\operatorname{ord}(0)=\infty$.
Since $p$ is odd, a subgroup $\mathcal{O}^{\times2}$ of $\mathcal{O}^\times$ has index 2.
Fix an element $\xi$ in $\mathcal{O}^\times \setminus \mathcal{O}^{\times2}$, so that $\mathcal{O}^\times=\mathcal{O}^{\times2}\sqcup \xi\mathcal{O}^{\times2}$.
Let $dx$ be the Haar measure on $F$ such that $\int_\mathcal{O} dx=1$.
We shall write $(-,-)_2$ for the Hilbert symbol of degree 2 over $F$.

For a nontrivial additive character $\psi\colon F\to \C^\times$, let $c(\psi)$ denote its conductor, i.e., the minimal integer $c$ such that $\psi$ is trivial on $\mathfrak{p}^c$.
For a nontrivial character $\eta \colon \mathcal{O}^\times \to \C^\times$, let $c(\eta)$ denote its conductor, i.e., the minimal positive integer $c$ such that $\eta$ is trivial on $1+\mathfrak{p}^c$.
For the trivial character $1$ of $\mathcal{O}^\times$, its conductor $c(1)$ is defined to be $0$.
For a character $\mu$ of $F^\times$, we write $c(\mu)=c(\mu|_{\mathcal{O}^\times})$ and call it the conductor of $\mu$.
Especially, $\mu$ is said to be unramified if $c(\mu)=0$.
For any $a\in F^\times$, let $\chi_a$ denote a quadratic or trivial character of $F^\times$ given by $x\mapsto (x,a)_2$.
It is known that the assignment $a\mapsto \chi_a$ gives a bijection between $F^\times/F^{\times2}$ and the set of quadratic or trivial characters of $F^\times$.

For a nontrivial additive character $\psi$ of $F$, let $\gamma_F(\psi)$ denote the (unnormalized) Weil index of $F\ni x \mapsto \psi(x^2)\in \C^\times$, which is an eighth root of unity in $\C^\times$.
For $a \in F^\times$, we shall define an additive character $\psi_a$ by $\psi_a(x)=\psi(ax)$, and put $\gamma_F(a,\psi)=\gamma_F(\psi_a)/\gamma_F(\psi)$, which we shall call the (normalized) Weil index.
Recall from \cite[Appendix]{rao} that
\begin{align*}
     & \gamma_F(ac^2,\psi)  =\gamma_F(a,\psi),                       &
     & \gamma_F(ab,\psi)   =\gamma_F(a,\psi)\gamma_F(b,\psi)(a,b)_2, & \\
     & \gamma_F(a,\psi_c)=(a,c)_2 \gamma_F(a,\psi),                  &
     & \gamma_F(-1,\psi)  =\gamma_F(\psi)^{-2},                      &
\end{align*}
for any $a,b,c\in F^\times$.
They and \cite[Lemma 1.6]{szp} imply that
\begin{equation*}
    \gamma_F(ac, \psi) = \gamma_F(c,\psi),
\end{equation*}
for any $a\in\mathcal{O}^\times$ and $c\in F^\times$ such that the parity of $\operatorname{ord}(c)$ coincides with that of $c(\psi)$.

For an abstract group $G$ and its elements $g$ and $h$, we shall write $\prescript{g}{}{h}=ghg^{-1}$ and $h^g=g^{-1}hg$.
In addition, for a subset $X$ of $G$, we write $\prescript{g}{}{X}$ (resp. $X^g$) for the set of elements of the form $\prescript{g}{}{x}$ (resp. $x^g$), where $x\in X$.
Moreover, for a map $\Psi$ from $X$, we write $\prescript{g}{}{\Psi}$ (resp. $\Psi^g$) for the map $\prescript{g}{}{X}\ni y\mapsto\Psi(y^g)$ (resp. $X^g\ni y\mapsto\Psi(\prescript{g}{}{y})$).

In this paper, every representation is a $\C$-representation.
Moreover, every representation of a locally profinite group is smooth and admissible.
For a topological group $G$, its subgroup $H$, and a representation $(\sigma,W)$ of $H$, we write $\Ind^G_H(\sigma)$ for the induced representation, i.e., $\Ind^G_H(\sigma)=(\pi,V)$, where $V$ is the set of locally constant functions $f \colon G\to W$ such that $f(hg)=\sigma(h)f(g)$ for any $h\in H$ and $g\in G$, and $[\pi(g)f](x)=f(xg)$, for any $g,x\in G$.
We shall also write $\cInd^G_H(\sigma)$ for the subrepresentation of $\Ind^G_H(\sigma)$ consisting of the functions $f\in\Ind^G_H(\sigma)$ such that the support of $f$ is compact modulo $H$.
Conversely, for any representation $(\pi,V)$ of $G$, we write $\Res^G_H(\pi)$ for the restriction of $(\pi,V)$ to $H$.

\subsection*{Acknowledgement}
The author would like to thank Atsushi Ichino for suggesting a problem on local newforms for metaplectic groups.
The author also would like to thank Shunsuke Yamana and Hiraku Atobe for many helpful comments.
Yuki Yamamoto explained to the author an argument in the proof of Lemma \ref{lem:Yamamotowhittakersupercuspidal}.
The author is grateful to him.
This work is supported by JSPS Research Fellowships for Young Scientists KAKENHI Grant Number 23KJ1824.
This work was partly supported by MEXT Promotion of Distinctive Joint Research Center Program JPM XP0619217849.

\section{Variants of Gauss sum}\label{secgausssum}
In this section, we shall calculate the following two variants of Gauss sum, which will be needed later.
For any nontrivial additive character $\psi$ of $F$ and any character $\chi$ of $F^\times$, we put
\begin{equation*}
    g(\chi,\psi) = \int_{\mathcal{O}^\times} \chi(x) \psi(x) dx,
\end{equation*}
and
\begin{equation*}
    h(\chi,\psi) = \int_{\mathcal{O}^\times} \chi(x) \psi(x^2) dx.
\end{equation*}
\begin{lemma}\label{lemgausssum}
    \begin{enumerate}[(1)]
        \item If $c(\chi)=0$, then we have
              \begin{align*}
                  g(\chi,\psi)=\begin{cases*}
                                   1-q^{-1}, & if $c(\psi)\leq0$, \\
                                   -q^{-1},  & if $c(\psi)=1$,    \\
                                   0,        & if $c(\psi)\geq2$.
                               \end{cases*}
              \end{align*}
        \item Assume that $c(\chi)\geq1$. Then
              \begin{align*}
                  |g(\chi,\psi)|=\begin{cases*}
                                     q^{-\frac{c(\psi)}{2}}, & if $c(\psi)=c(\chi)$,     \\
                                     0,                      & if $c(\psi)\neq c(\chi)$.
                                 \end{cases*}
              \end{align*}
              In particular, $g(\chi,\psi)\neq0$ if and only if $c(\chi)=c(\psi)$.
    \end{enumerate}
\end{lemma}
\begin{proof}
    The proof is similar to that of the evaluation of classical Gauss sums.
    If $c(\psi)=0$, then
    \begin{equation*}
        g(\chi,\psi)
        =\int_{\mathcal{O}^\times} \psi(x) dx
        =\int_{\mathcal{O}} \psi(x) dx - \int_{\mathfrak{p}} \psi(x) dx.
    \end{equation*}
    This proves the assertion (1).
    Now suppose that $c(\chi)\geq1$.
    If $c(\psi)\leq0$, it is easy to see that $g(\chi,\psi)=0$.
    Suppose that $c(\psi)\geq1$.
    For any $a\in\mathcal{O}^\times$, we have $g(\chi,\psi_a)=\chi(a)^{-1} g(\chi,\psi)$.
    Thus
    \begin{equation*}
        \int_{\mathcal{O}^\times} g(\chi,\psi_a)\overline{g(\chi,\psi_a)} da
        = \int_{\mathcal{O}^\times} |g(\chi,\psi)|^2 da
        = \operatorname{vol}(\mathcal{O}^\times) |g(\chi,\psi)|^2.
    \end{equation*}
    On the other hand,
    \begin{align*}
        \int_{\mathcal{O}^\times} g(\chi,\psi_a)\overline{g(\chi,\psi_a)} da
         & = \int_{\mathcal{O}^\times} \int_{\mathcal{O}^\times} \int_{\mathcal{O}^\times} \chi(x) \psi(ax) \overline{\chi(y)\psi(ay)} dxdyda \\
         & = \int_{\mathcal{O}^\times} \int_{\mathcal{O}^\times} \chi(xy^{-1}) \left( \int_{\mathcal{O}^\times} \psi(a(x-y)) da \right) dxdy.
    \end{align*}
    Changing the variables $x\mapsto u=xy^{-1}$ and $a\mapsto t=ay$, this equals
    \begin{align*}
         & \int_{\mathcal{O}^\times} \int_{\mathcal{O}^\times} \chi(u) \left( \int_{\mathcal{O}^\times} \psi(t(u-1)) dt \right) dudy                                                     \\
         & = \int_{\mathcal{O}^\times} \left( \sum_{\ell=0}^{\infty} \int_{U^\ell \setminus U^{\ell+1}} \chi(u) \left( \int_{\mathcal{O}^\times} \psi(t(u-1)) dt \right) du \right)  dy,
    \end{align*}
    where we put $U^0=\mathcal{O}^\times$ and $U^\ell=1+\mathfrak{p}^\ell$ for $\ell\geq1$.
    For any $u\in U^\ell \setminus U^{\ell+1}$, we have $\operatorname{ord}(u-1)=\ell$ and hence
    \begin{align*}
        \int_{\mathcal{O}^\times} \psi(t(u-1)) dt
        =\begin{cases*}
             0,        & if $\ell \leq c(\psi)-2$, \\
             -q^{-1},  & if $\ell= c(\psi)-1$,     \\
             1-q^{-1}, & if $\ell \geq c(\psi)$.
         \end{cases*}
    \end{align*}
    Thus, the integral above is equal to
    \begin{align*}
         & \int_{\mathcal{O}^\times} \left( -q^{-1}\int_{U^{c(\psi)-1} \setminus U^{c(\psi)}} \chi(u) du + (1-q^{-1}) \sum_{\ell=c(\psi)}^{\infty} \int_{U^\ell \setminus U^{\ell+1}} \chi(u) du \right)  dy \\
         & = \operatorname{vol}(\mathcal{O}^\times) \left( -q^{-1} \int_{U^{c(\psi)-1}} \chi(u) du +q^{-1} \int_{U^{c(\psi)}} \chi(u) du + (1-q^{-1}) \int_{U^c(\psi)} \chi(u) du \right)                    \\
         & =\operatorname{vol}(\mathcal{O}^\times) \left( -q^{-1} \int_{U^{c(\psi)-1}} \chi(u) du + \int_{U^{c(\psi)}} \chi(u) du \right).
    \end{align*}
    Consequently, we have
    \begin{equation*}
        |g(\chi,\psi)|^2
        = -q^{-1} \int_{U^{c(\psi)-1}} \chi(u) du + \int_{U^{c(\psi)}} \chi(u) du.
    \end{equation*}
    Now the assertion (2) follows.
\end{proof}
Next we consider $h(\chi,\psi)$.
Note that $h(\chi,\psi)=0$ if $\chi(-1)=-1$.
\begin{lemma}\label{lemgausssum2}
    Assume that $\chi(-1)=+1$.
    \begin{enumerate}[(1)]
        \item If $c(\psi)\leq 0$, then we have
              \begin{align*}
                  h(\chi,\psi)=\begin{cases*}
                                   1-q^{-1}, & if $c(\chi)=0$,    \\
                                   0,        & if $c(\chi)\geq1$.
                               \end{cases*}
              \end{align*}
        \item Assume that $c(\psi) \geq 1$. Then we have
              \begin{align*}
                  |h(\chi,\psi)|^2 + |h(\chi,\psi_\xi)|^2
                  = \begin{cases*}
                        4q^{-c(\psi)},   & if $c(\chi)=c(\psi)$,           \\
                        2q^{-1}+2q^{-2}, & if $c(\chi)=0$ and $c(\psi)=1$, \\
                        0,               & otherwise.
                    \end{cases*}
              \end{align*}
              In particular, at least one of $h(\chi,\psi)$ or $h(\chi,\psi_\xi)$ is nonzero if and only if $c(\chi)=c(\psi)$ or $(c(\chi),c(\psi))=(0,1)$.
    \end{enumerate}
\end{lemma}
\begin{proof}
    The proof is similar to that of Lemma \ref{lemgausssum}.
    If $c(\psi)\leq0$, then we have
    \begin{equation*}
        h(\chi,\psi) = \int_{\mathcal{O}} \chi(x) dx - \int_{\mathfrak{p}} \chi(x) dx.
    \end{equation*}
    This proves the assertion (1).
    Suppose that $c(\psi)\geq1$.
    For any $\alpha\in \mathcal{O}^\times$, we have $h(\chi,\psi_{\alpha^2})=\chi(\alpha)^{-1} h(\chi,\psi)$ and $h(\chi,\psi_{\xi\alpha^2})=\chi(\alpha)^{-1} h(\chi,\psi_\xi)$.
    Since $\mathcal{O}^\times=\mathcal{O}^{\times2} \sqcup \xi\mathcal{O}^{\times2}$, we have
    \begin{align*}
        \int_{\mathcal{O}^\times} h(\chi,\psi_a)\overline{h(\chi,\psi_a)} da
         & = \int_{\mathcal{O}^{\times2}} |h(\chi,\psi)|^2 da + \int_{\xi\mathcal{O}^{\times2}} |h(\chi,\psi_\xi)|^2 da \\
         & =\frac{1}{2} \operatorname{vol}(\mathcal{O}^\times) \left( |h(\chi,\psi)|^2 + |h(\chi,\psi_\xi)|^2 \right).
    \end{align*}
    On the other hand,
    \begin{align*}
        \int_{\mathcal{O}^\times} h(\chi,\psi_a)\overline{h(\chi,\psi_a)} da
         & = \int_{\mathcal{O}^\times} \int_{\mathcal{O}^\times} \int_{\mathcal{O}^\times} \chi(x) \psi(ax^2) \overline{\chi(y)\psi(ay^2)} dxdyda \\
         & = \int_{\mathcal{O}^\times} \int_{\mathcal{O}^\times} \chi(xy^{-1}) \left( \int_{\mathcal{O}^\times} \psi(a(x^2-y^2)) da \right) dxdy.
    \end{align*}
    Changing the variables $x\mapsto u=xy^{-1}$ and $a\mapsto t=ay^2$, this equals
    \begin{align*}
         & \int_{\mathcal{O}^\times} \int_{\mathcal{O}^\times} \chi(u) \left( \int_{\mathcal{O}^\times} \psi(t(u^2-1)) dt \right) dudy                                                   \\
         & =\int_{\mathcal{O}^\times} \left( \sum_{\ell=0}^{\infty} \int_{V^\ell \setminus V^{\ell+1}} \chi(u) \left( \int_{\mathcal{O}^\times} \psi(t(u^2-1)) dt \right) du \right) dy,
    \end{align*}
    where we put $V^0=\mathcal{O}^\times$ and $V^\ell=\{\pm1\}+\mathfrak{p}^\ell$ for $\ell\geq1$.
    For any $u\in V^\ell \setminus V^{\ell+1}$, we have $\operatorname{ord}(u^2-1)=\ell$.
    Thus, the integral above is equal to
    \begin{align*}
         & \int_{\mathcal{O}^\times} \left( -q^{-1}\int_{V^{c(\psi)-1} \setminus V^{c(\psi)}} \chi(u) du + (1-q^{-1}) \sum_{\ell=c(\psi)}^{\infty} \int_{V^\ell \setminus V^{\ell+1}} \chi(u) du \right)  dy \\
         & =\operatorname{vol}(\mathcal{O}^\times) \left( -q^{-1} \int_{V^{c(\psi)-1}} \chi(u) du +q^{-1} \int_{V^{c(\psi)}} \chi(u) du + (1-q^{-1}) \int_{V^c(\psi)} \chi(u) du \right)                     \\
         & =\operatorname{vol}(\mathcal{O}^\times) \left( -q^{-1} \int_{V^{c(\psi)-1}} \chi(u) du + \int_{V^{c(\psi)}} \chi(u) du \right).
    \end{align*}
    Consequently, we have
    \begin{align*}
        \frac{1}{2} \left( |h(\chi,\psi)|^2 + |h(\chi,\psi_\xi)|^2 \right)
        = -q^{-1} \int_{V^{c(\psi)-1}} \chi(u) du + \int_{V^{c(\psi)}} \chi(u) du.
    \end{align*}
    The assertion (2) follows.
\end{proof}
\begin{remark}
    If $c(\chi)=c(\psi)\geq2$, one can show that exactly one of $|h(\chi,\psi)|$ or $|h(\chi,\psi_\xi)|$ is $2q^{-\frac{c(\psi)}{2}}$ and the other is $0$, by evaluating an integral
    \begin{equation*}
        \int_{1+\mathfrak{p}} h(\chi,\psi_a)\overline{h(\chi,\psi_a)} da,
    \end{equation*}
    in two ways as above.
\end{remark}

\section{The metaplectic group}\label{secmetap}
In this section, we shall establish some notations for the metaplectic group and its subgroups, and review several notions concerning its representations.
\subsection{The metaplectic group}
Put $G=\SL_2(F)$, and let $\widetilde{G}=\widetilde{\SL_2}(F)$ be the nonlinear double cover of $G$.
As a set, we shall write $\widetilde{G}=G\times\{\pm1\}$ with the group law $(g_1,\epsilon_1)\cdot(g_2,\epsilon_2)=(g_1g_2,\epsilon_1\epsilon_2\bm{c}(g_1,g_2))$, where
\begin{align*}
    \bm{c}(g_1,g_2)                      =\left( \frac{\bm{x}(g_1)}{\bm{x}(g_1g_2)}, \frac{\bm{x}(g_2)}{\bm{x}(g_1g_2)} \right)_2, \\
    \bm{x}\left(\left( \begin{array}{cc}
                                   a & b \\c&d
                               \end{array} \right)\right) = \begin{cases*}
                                                        c, & if $c\neq0$, \\
                                                        d, & if $c=0$.
                                                    \end{cases*}
\end{align*}
The 2-cocycle $\bm{c}$ is called the Kubota cocycle.
For any subset $A\subset G$, we shall write $\widetilde{A}$ for the preimage of $A$ in $\widetilde{G}$.
A function $f\colon \widetilde{A} \to \C$ is said to be genuine if $f((a,-1))=-f((a,1))$ for any $a\in A$.
Analogously, when $A$ is a subgroup of $G$, a representation of $\widetilde{A}$ is said to be genuine if it does not factor through the covering map $\widetilde{A}\to A$.

Put
\begin{align*}
     & t(a)=\left( \begin{array}{cc}a & \\ &a^{-1} \end{array} \right), &  & n(b)=\left( \begin{array}{cc}1&b\\ &1\end{array} \right), &  & n^{\mathrm{op}}(b)=\left( \begin{array}{cc}1& \\ b&1\end{array} \right), &  & \beta=\left( \begin{array}{cc}1&\\ &\varpi\end{array} \right), &  & w=\left( \begin{array}{cc}&1\\ -1&\end{array} \right), &
\end{align*}
for $a\in F^\times$ and $b\in F$.
Put
\begin{align*}
     & T=\Set{t(a)|a\in F^\times}, &  & N=\Set{n(b)|b\in F}, &  & B=TN, &
\end{align*}
\begin{gather*}
    K^0 =\SL_2(\mathcal{O}),\quad K^1=\Set{\left( \begin{array}{cc}a&b \\ c&d \end{array} \right) \in \SL_2(F) | a,d\in \mathcal{O},\ b\in \mathfrak{p}^{-1},\ c\in \mathfrak{p}},\\
    K^0_m=\Set{\left( \begin{array}{cc}a&b \\ c&d \end{array} \right) \in K^0 | c\in \mathfrak{p}^m},\quad K^1_m=\Set{\left( \begin{array}{cc}a&b \\ c&d \end{array} \right) \in K^1 | c\in \mathfrak{p}^{m+1}}.
\end{gather*}
Note that $K^0=K^0_0$ and $K^1_m=\prescript{\beta}{}{K^0_m}=\beta K^0_m \beta^{-1}$.
We shall regard $N$ as a subgroup of $\widetilde{G}$ via a splitting $N\ni n\mapsto(n,1)\ni \widetilde{G}$.
Fix additive characters $\psi^0$ and $\psi^1$ of $F$ with the conductors $0$ and $-1$, respectively.
Then it is well-known that $K^0$ and $K^1$ are maximal compact subgroups of $G$, and that they split the covering map $\widetilde{G}\to G$ uniquely.
We shall write $x\mapsto(x,\bm{s}^0(x))$ and $x\mapsto(x,\bm{s}^1(x))$ for the splittings over $K^0$ and $K^1$, respectively.
The splittings are given explicitly by
\begin{align*}
    \bm{s}^\varepsilon(n(b))                    & =1,                            & b & \in\mathfrak{p}^{-\varepsilon}, \\
    \bm{s}^\varepsilon(t(a))                    & =\gamma_F(a,\psi^\varepsilon), & a & \in\mathcal{O}^\times,          \\
    \bm{s}^\varepsilon(n^{\mathrm{op}}(c))      & =1,                            & c & \in\mathfrak{p}^\varepsilon,    \\
    \bm{s}^\varepsilon(w t(\varpi^\varepsilon)) & = 1,                           &   &
\end{align*}
where $\varepsilon\in\{0,1\}$.
See \cite[Lemma 1.1]{hi} for the proof.
Via these splittings, we regard $K^0_m$ and $K^1_m$ as subgroups of $\widetilde{G}$.

We have one lemma, which will be needed later.
\begin{lemma}\label{lemcoset}
    Let $m\geq0$ be an integer.
    We have a natural bijection
    \begin{equation*}
        K^0_m\backslash K^0/ B\cap K^0 \to K^0_m\backslash G/ B,\quad K^0_m g(B\cap K^0) \mapsto K^0_m g B,
    \end{equation*}
    and the set
    \begin{align*}
        \begin{cases*}
            \{1_2\},                                                                                            & when $m=0$,    \\
            \{1_2, w\},                                                                                         & when $m=1$,    \\
            \set{1_2, w} \cup \set{n^{\mathrm{op}}(\varpi^i), n^{\mathrm{op}}(\xi \varpi^i) | i=1,\ldots, m-1}, & when $m\geq2$,
        \end{cases*}
    \end{align*}
    is a complete representative system of the double coset spaces.
\end{lemma}
\begin{proof}
    The natural bijection of double coset spaces is implied by the Iwasawa decomposition $G=K^0B$.

    Now we consider the assertion on representatives.
    If $m=0$, the assertion is trivial.
    Suppose that $m\geq1$.
    Let $g=(g_{ij})\in K^0$.
    If $g_{21}\in\mathcal{O}^\times$, we have
    \begin{equation*}
        t(-g_{21})n(-g_{21}^{-1}g_{11}) \cdot g \cdot n(-g_{21}^{-1}g_{22})=w.
    \end{equation*}
    Since $t(-g_{21})n(-g_{21}^{-1}g_{11})\in K^0_m$ and $n(-g_{21}^{-1}g_{22})\in B\cap K^0$, this means that
    \begin{equation*}
        K^0_m g (B\cap K^0) = K^0_m w (B\cap K^0).
    \end{equation*}
    Suppose that $g_{21}\in \mathfrak{p}$.
    If $g_{21}\in \mathfrak{p}^m$, then $g\in K^0_m$ and we have
    \begin{equation*}
        K^0_m g (B\cap K^0) = K^0_m 1_2 (B\cap K^0).
    \end{equation*}
    If $g_{21}\in \mathfrak{p}\setminus \mathfrak{p}^m$, then $g_{11} \in\mathcal{O}^\times$, and there exist $\delta\in\{0,1\}$, $a\in\mathcal{O}^\times$, and $1\leq i \leq m-1$ such that $g_{21}g_{11}^{-1}=\xi^\delta a^2\varpi^i$.
    Then we have
    \begin{equation*}
        t(a) \cdot g \cdot n(-g_{11}^{-1}g_{12}) t(g_{11}^{-1}) t(a^{-1}) = n^{\mathrm{op}}(\xi^\delta \varpi^i),
    \end{equation*}
    which means that
    \begin{equation*}
        K^0_m g (B\cap K^0) = K^0_m n^{\mathrm{op}}(\xi^\delta \varpi^i) (B\cap K^0).
    \end{equation*}
    Thus we have
    \begin{align*}
        K^0= K^0_m w (B\cap K^0) \cup K^0_m 1_2 (B\cap K^0) \cup \bigcup_{i=1}^{m-1} \left( K^0_m n^{\mathrm{op}}(\varpi^i) (B\cap K^0) \cup K^0_m n^{\mathrm{op}}(\xi\varpi^i) (B\cap K^0) \right).
    \end{align*}
    If $m=1$, it is of course understood to be a union of the former two cosets.
    Note that if the (2,1)-component of $g\in K^0$ is an element of $\varpi^i \mathcal{O}^\times$ and $0\leq i\leq m-1$, then so is that of every element in $K^0_m g (B\cap K^0)$.
    Hence we have
    \begin{align*}
        K^0= K^0_m w (B\cap K^0) \sqcup K^0_m 1_2 (B\cap K^0) \sqcup \bigsqcup_{i=1}^{m-1} \left( K^0_m n^{\mathrm{op}}(\varpi^i) (B\cap K^0) \cup K^0_m n^{\mathrm{op}}(\xi\varpi^i) (B\cap K^0) \right).
    \end{align*}
    Now, it remains to show that if $1\leq i \leq m-1$, then $n^{\mathrm{op}}(\varpi^i)$ and $n^{\mathrm{op}}(\xi\varpi^i)$ are not contained in a same coset.
    Suppose that they are.
    Then there exist
    \begin{equation*}
        \left( \begin{array}{cc}a&b\\ c&d \end{array} \right)\in K^0_m \quad \text{and} \quad \left( \begin{array}{cc}x&y\\ &x^{-1} \end{array} \right)\in B\cap K^0,
    \end{equation*}
    such that
    \begin{align*}
        \left( \begin{array}{cc}a&b\\ c&d \end{array} \right) n^{\mathrm{op}}(\varpi^i) = n^{\mathrm{op}}(\xi\varpi^i) \left( \begin{array}{cc}x&y\\ &x^{-1} \end{array} \right).
    \end{align*}
    In other words, there exist $a, d, x\in\mathcal{O}^\times$, $b, y\in \mathcal{O}$, and $c\in\mathfrak{p}^m$, such that
    \begin{align*}
        \left\{
        \begin{aligned}
             & ad-bc=1,                  \\
             & a+\varpi^ib=x,            \\
             & b=y,                      \\
             & c+\varpi^id=\xi\varpi^ix, \\
             & d=\xi\varpi^iy+x^{-1}.
        \end{aligned}
        \right.
    \end{align*}
    Combining the assumption $1\leq i \leq m-1$, these imply that $\xi\in d^2+\mathfrak{p}$.
    Since the residual characteristic $p$ is odd, we have $d^2 + \mathfrak{p} \subset \mathcal{O}^{\times2}$.
    This is a contradiction.
\end{proof}
Similarly, we have the following.
\begin{corollary}\label{corcoset}
    The set
    \begin{align*}
        \begin{cases*}
            \{1_2\},                                                                                                    & when $m=0$,    \\
            \{1_2, w\},                                                                                                 & when $m=1$,    \\
            \set{1_2, w} \cup \set{n^{\mathrm{op}}(\varpi^{i+1}), n^{\mathrm{op}}(\xi \varpi^{i+1}) | i=1,\ldots, m-1}, & when $m\geq2$,
        \end{cases*}
    \end{align*}
    is a complete representative system of $K^1_m\backslash G/ B$.
\end{corollary}
\begin{proof}
    A bijection $g\mapsto \beta g\beta^{-1}$ from $G$ to $G$ gives a bijection between $K^0_m\backslash G/ B$ and $K^1_m\backslash G/ B$.
    We have
    \begin{align*}
         & \beta 1_2 \beta^{-1}=1_2, &  & \beta w \beta^{-1}=wt(\varpi), &  & \beta n^{\mathrm{op}}(y)\beta^{-1}=n^{\mathrm{op}}(\varpi y). &
    \end{align*}
    Then the assertion follows from Lemma \ref{lemcoset}.
\end{proof}

\subsection{Representations of the metaplectic group}\label{secrep}
Let $(\pi,V)$ be an irreducible admissible representation of $\widetilde{G}$, and $\Psi$ a nontrivial additive character of $F$.
The representation $(\pi,V)$ is said to be $\Psi$-generic if
\begin{equation*}
    \Hom_N(\pi,\Psi) \neq\{0\},
\end{equation*}
i.e., there exists a nonzero linear map $\ell_\Psi \colon V\to \C$ such that $\ell_\Psi(\pi(n(b))v)=\Psi(b)\ell_\Psi(v)$, for any $b\in F$ and $v\in V$.
Any nonzero element $\ell_\Psi\in \Hom_N(\pi,\Psi)$ is called a $\Psi$-Whittaker functional, and it is proven in \cite{wal1} that the space $\Hom_N(\pi,\Psi)$ of $\Psi$-Whittaker functionals is at most one dimensional.

The representation $(\pi,V)$ is said to be supercuspidal if for any $v\in V$ there exists a compact open subgroup $N(v)$ of $N$ such that
\begin{equation*}
    \int_{N(v)} \pi(n)v dn=0,
\end{equation*}
where $dn$ is a Haar measure on $N(v)$.
We shall review a classification of irreducible supercuspidal representations of $\widetilde{G}$ later in \S\ref{sectype}.

Now we shall review the Weil representations of $\widetilde{G}$.
For every nontrivial additive character $\psi$ of $F$, let $(\omega_\psi, \mathcal{S}(F))$ be the Schr\"{o}dinger model of the Weil representation of $\widetilde{G}$ associated to $\psi$, i.e., $\mathcal{S}(F)$ is the space of locally constant and compactly supported functions on $F$ and the action $\omega_\psi$ is determined by
\begin{align*}
    [\omega_\psi((t(a),\epsilon))\cdot\varphi](y) & =\epsilon|a|^{\frac{1}{2}}\gamma_F(a,\psi)^{-1} \varphi(ay), \\
    [\omega_\psi(n(b)) \cdot \varphi](y)          & =\psi(by^2)\varphi(y),                                       \\
    [\omega_\psi((w,1))\cdot\varphi](y)           & =\gamma_F(\psi)\int_F \varphi(x)\psi(2xy)d_{\psi_2} x,
\end{align*}
where $d_{\psi_2}x$ denotes the self-dual measure on $F$ with respect to $\psi_2$.
Note that $d_{\psi_2}x=q^{\frac{c(\psi_2)}{2}}dx$, and that $c(\psi_2)=c(\psi)$ since $p\neq2$.
Let $\mathcal{S}^+(F)$ (resp. $\mathcal{S}^-(F)$) be the subspace of $\mathcal{S}(F)$ consisting of all even (resp. odd) functions, i.e., the functions $\varphi\in\mathcal{S}(F)$ such that $\varphi(-x)=\varphi(x)$ (resp. $\varphi(-x)=-\varphi(x)$).
Then a decomposition $\mathcal{S}(F)=\mathcal{S}^+(F)\oplus\mathcal{S}^-(F)$ gives the decomposition of the Weil representation, which we shall write $\omega_\psi=\omega_\psi^+ \oplus \omega_\psi^-$.
The representation $\omega_\psi^+$ (resp. $\omega_\psi^-$) is called the even (resp. odd) Weil representation.
The even and odd Weil representations are irreducible.
Note that $\omega_\psi\cong\omega_{\psi_a}$ if and only if $a\in F^{\times2}$.
Thus, for any quadratic or trivial character $\chi=\chi_a$, we shall write $\omega_{\psi,\chi}=\omega_{\psi_a}$, $\omega_{\psi,\chi}^+=\omega_{\psi_a}^+$, and $\omega_{\psi,\chi}^-=\omega_{\psi_a}^-$.
We note:
\begin{itemize}
    \item $\omega_{\psi_a}^\pm$ is $\psi_b$-generic if and only if $ab^{-1}\in F^{\times2}$;
    \item $\omega_{\psi_a}^+$ is not supercuspidal, and $\omega_{\psi_a}^-$ is supercuspidal;
    \item $\omega_{\psi_a}^\pm$ is characterized by these properties.
\end{itemize}

Next we define principal series representations of $\widetilde{G}$.
For any nontrivial additive character $\psi$ of $F$, we have a genuine character $\chi_\psi$ of $\widetilde{T}$ defined by
\begin{align*}
     & \chi_\psi((t(a),\epsilon))=\epsilon\gamma_F(a,\psi)^{-1}, &  & a\in F^\times,\ \epsilon\in\{\pm1\}. &
\end{align*}
Given a character $\mu$ of $F^\times$, we shall write $\pi_\psi(\mu)$ for the representation parabolically induced from $\widetilde{B}$ by $\mu \cdot \chi_\psi$, and realize it on the space of locally constant functions $f\colon \widetilde{G}\to \C$ such that
\begin{equation*}
    f((t(a),\epsilon)(n(b),1)g)=|a|\mu(a)\chi_\psi((t(a),\epsilon))f(g),
\end{equation*}
for any $a\in F^\times$, $\epsilon\in\{\pm1\}$, $b\in F$, and $g\in\widetilde{G}$.
The representation is defined by $[g\cdot f](x)=f(xg)$.
\begin{remark}
    There exists $z\in F^\times$ such that $\psi^1=\psi^0_z$.
    Then we have $\gamma_F(a,\psi^1)=\gamma_F(a,\psi^0)\chi_z(a)$, and hence $\pi_{\psi^1}(\mu)=\pi_{\psi^0}(\mu\chi_z)$.
\end{remark}
The reducibility property of these representations are summarized as follows.
\begin{proposition}\label{propreducibility}
    \begin{enumerate}[(1)]
        \item The representation $\pi_\psi(\mu)$ is irreducible if and only if $\mu^2 \neq |-|^{\pm1}$. In this case we have $\pi_\psi(\mu)\cong\pi_\psi(\mu^{-1})$.
        \item If $\mu=\chi \cdot |-|^{\frac{1}{2}}$, where $\chi$ is a quadratic or trivial character, then $\pi_\psi(\mu)$ is reducible and there are an irreducible representation $\mathit{St}_{\psi,\chi}$ and a short exact sequence
              \begin{equation*}
                  0 \to \mathit{St}_{\psi,\chi} \to \pi_\psi(\mu) \to \omega_{\psi,\chi}^+ \to 0.
              \end{equation*}
              We shall call the representation $\mathit{St}_{\psi,\chi}$ the Steinberg representation associated to $\psi$ and $\chi$.
        \item If $\mu= \chi \cdot |-|^{-\frac{1}{2}}$, where $\chi$ is a quadratic or trivial character, then $\pi_\psi(\mu)$ is reducible and there is a short exact sequence
              \begin{equation*}
                  0 \to \omega_{\psi,\chi}^+ \to \pi_\psi(\mu) \to \mathit{St}_{\psi,\chi} \to 0.
              \end{equation*}
    \end{enumerate}
\end{proposition}
It is known that the proposition gives all the irreducible non-supercuspidal genuine representations of $\widetilde{G}$.

\subsection{Definition of conductors}
Let $(\pi,V)$ be an admissible representation of $\widetilde{G}$, and $\psi$ a nontrivial additive character of $F$.
Consider an element $-1_\psi=(-1_2, \gamma_F(-1,\psi))$ in $\widetilde{G}$.
It is central in $\widetilde{G}$ and has order 2.
Hence it acts on $V$ by $\pm1$.
We define the central sign $z_\psi(\pi)\in\{\pm1\}$ with respect to $\psi$ by $\pi(-1_\psi)=z_\psi(\pi)1_V$.

Let $\eta$ be any character of $\mathcal{O}^\times$.
When $m\geq c(\eta)$, $\eta$ gives characters of $K^0_m$ and $K^1_m$ both defined by
\begin{align*}
    \left( \begin{array}{cc}a&b \\ c&d \end{array} \right) \mapsto \eta(d),
\end{align*}
for which we shall write $\eta$.
Let $(\pi,V)$ be an admissible representation of $\widetilde{G}$.
For $K_m=K^0_m$ or $K^1_m$, we put
\begin{align*}
    \pi^{K_m}_\eta=V^{K_m}_\eta
    =\begin{cases*}
         \Set{v\in V | \pi(x)v=\eta(x)v,\ \text{for all $x\in K_m$}}, & if $m \geq c(\eta)$, \\
         \{0\},                                                       & if $m<c(\eta)$.
     \end{cases*}
\end{align*}
Let $\varepsilon=0$ or $1$.
Let $\eta$ be a character of $\mathcal{O}^\times$ such that $\eta(-1)=z_{\psi^\varepsilon}(\pi)$.
We put
\begin{align*}
    c^\varepsilon_\eta(\pi)
    =\min(m\geq0 \mid \pi^{K^\varepsilon_m}_\eta \neq \{0\}).
\end{align*}
We also define $c^\varepsilon_{\min}(\pi)=\min(c^\varepsilon_\eta(\pi) \mid \eta)$, where $\eta$ runs over characters of $\mathcal{O}^\times$ such that $\eta(-1)=z_{\psi^\varepsilon}(\pi)$.
In this paper, we calculate the dimensions of $\pi^{K_m}_\eta$ and $c^\varepsilon_\eta(\pi)$ for irreducible genuine representations $\pi$ of $\widetilde{G}$.

\subsection{The result of Roberts--Schmidt}\label{subsec:Roberts--Schmidt}
We shall recall the result of Roberts--Schmidt \cite{rs-Mp2}.
Let $\varepsilon\in\{0,1\}$.
Put $\psi=\psi^\varepsilon$ and $K_m=K^\varepsilon_m$.
For any admissible representation $(\pi, V)$ of $\widetilde{G}$, Roberts--Schmidt defined a linear automorphism $\alpha_2$ on $V$ by
\begin{equation*}
    v \mapsto \alpha_2 v = \pi(\left( \begin{array}{cc}\varpi^{-1}&\\ &\varpi \end{array} \right),1) v.
\end{equation*}
Let $\eta$ be a character of $\mathcal{O}^\times$ such that $\eta(-1)=z_{\psi^\varepsilon}(\pi)$.
It follows immediately that $\alpha_2(\pi^{K_m}_\eta) \subset \pi^{K_{m+2}}_\eta$.
In addition, $\pi^{K_m}_\eta$ is a subspace of $\pi^{K_{m+1}}_\eta$.
Following \cite{rs-Mp2}, we define the subspace $\pi^{K_m, \old}_\eta$ of local oldforms in $\pi^{K_m}_\eta$ as the linear subspace generated by $\pi^{K_{m-1}}_\eta$ and $\alpha_2(\pi^{K_{m-2}}_\eta)$, unless $m=0,1$.
In the cases of $m=0, 1$, we define $\pi^{K_0, \old}_\eta$ to be $\{0\}$, and $\pi^{K_1, \old}_\eta$ to be $\pi^{K_0}_\eta$.
Then we shall define
\begin{equation*}
    \pi^{K_m, \new}_\eta = \pi^{K_m}_\eta / \pi^{K_m, \old}_\eta.
\end{equation*}
Let $F_\psi(\pi)$ be the set of $a$ in $F^\times$ such that $\pi$ is $\psi_a$-generic.
The group $F^{\times2}$ acts on $F_\psi(\pi)$.
The number $\# F_\psi(\pi) / F^{\times2}$ is given by Waldspurger as follows.
\begin{proposition}\label{prop:numgeneric}
    Let $\pi$ be an irreducible genuine representation of $\widetilde{G}$.
    Then we have
    \begin{align*}
        \# F_\psi(\pi) / F^{\times2}
        =\begin{cases*}
             1, & if $\pi$ is an even or odd Weil representation,               \\
             2, & if $\pi$ is supercuspidal but not an odd Weil representation, \\
             3, & if $\pi$ is a Steinberg representation,                       \\
             4, & if $\pi$ is an irreducible principal series representation.
         \end{cases*}
    \end{align*}
\end{proposition}
\begin{theorem}\label{thm:rs-Mp2}
    Let $\pi$ be an irreducible genuine representation of $\widetilde{G}$, and $\eta$ a character of $\mathcal{O}^\times$ such that $\eta(-1)=z_{\psi^\varepsilon}(\pi)$.
    Then $\pi^{K_m, \new}_\eta$ is finite dimensional for all $m\geq0$, and zero-dimensional for all but finitely many $m\geq 0$, and we have
    \begin{equation*}
        \sum_{m=0}^\infty \dim_\C \pi^{K_m, \new}_\eta = \# F_\psi(\pi) / F^{\times2}.
    \end{equation*}
\end{theorem}
\begin{proof}
    If $\varepsilon=0$, this is the main theorem of \cite{rs-Mp2}.
    Even if $\varepsilon=1$, it can be proven similarly.
\end{proof}

\section{Non-supercuspidal representations}\label{sec:ps}
In this section, we shall consider irreducible non-supercuspidal genuine representations of $\widetilde{G}=\widetilde{\SL_2}(F)$.
Let $\varepsilon\in\{0,1\}$.
Put $\psi=\psi^\varepsilon$ and $K_m=K^\varepsilon_m$.
We shall fix them throughout this section.
\subsection{Principal series representations}\label{subsec:ps}
We first consider the principal series representations.
For a character $\mu$ of $F^\times$ and $\eta$ of $\mathcal{O}^\times$, by slight abuse of notation, we write $\eta\mu$ or $\mu\eta$ for the product of $\eta$ and $\mu|_{\mathcal{O}^\times}$.
\begin{theorem}\label{thmps}
    Let $\mu$ be a character of $F^\times$, and $\eta$ a character of $\mathcal{O}^\times$.
    \begin{enumerate}[(1)]
        \item Assume that $\eta(-1)\neq\mu(-1)$. Then $\pi_\psi(\mu)^{K_m}_\eta=\{0\}$ for all $m\geq0$.
        \item Assume that $\eta(-1)=\mu(-1)$. If $c(\mu)=c(\eta)=m=0$, then $\pi_\psi(\mu)^{K_m}_\eta=\pi_\psi(\mu)^{K_0}_1$ is 1-dimensional. Otherwise, we have
              \begin{align*}
                   & \dim_\C \pi_\psi(\mu)^{K_m}_\eta =                                                                                                                               \\
                   & \quad \begin{dcases*}
                               0,                                                                                                            & if $m< c(\mu)$,             \\
                               2( (m-c(\mu)-c(\eta\mu)+1)^+ + (m-c(\mu)-c(\eta\mu^{-1})+1)^+)-\delta_{c(\eta\mu)}-\delta_{c(\eta\mu^{-1})} , & if $c(\mu)\leq m< 2c(\mu)$, \\
                               2( m-c(\eta\mu)-c(\eta\mu^{-1})+1 )^+ -\delta_{c(\eta\mu)}-\delta_{c(\eta\mu^{-1})},                          & if $2c(\mu)\leq m$,
                           \end{dcases*}
              \end{align*}
              where $r^+$ denotes the positive part $\max(0,r)$, and $\delta_s$ does the Kronecker delta $\delta_{s,0}$.
    \end{enumerate}

    In particular, $c^\varepsilon_{\min}(\pi_\psi(\mu))=c^\varepsilon_\eta(\pi_\psi(\mu))$ if and only if $\eta=\mu^{\pm1}|_{\mathcal{O}^\times}$, and $c^\varepsilon_{\min}(\pi_\psi(\mu))=c(\mu)$.
\end{theorem}
\begin{proof}
    The proof is similar to that of \cite[Proposition 3.2.8]{lr}.

    Note that the central sign of $\pi_\psi(\mu)$ with respect to $\psi$ is equal to $\mu(-1)$.
    Hence we have $\pi_\psi(\mu)^{K_m}_\eta=\{0\}$ for all $m\geq0$, if $\mu(-1)\neq \eta(-1)$.
    Suppose that $\mu(-1)=\eta(-1)$.
    Recall that $\pi_\psi(\mu)^{K_m}_\eta$ is defined to be $\{0\}$ when $m<c(\eta)$.
    Suppose that $m\geq c(\eta)$.

    We first consider the case $\varepsilon=0$.
    By the proof of \cite[Corollary 2.7 (1)]{yam}, the space $\pi_\psi(\mu)^{K_m}_\eta$ is isomorphic to
    \begin{align*}
        \Hom_{K_m} (\eta, \Res^G_{K_m}(\Ind^{\widetilde{G}}_{\widetilde{B}}( |-| \mu \chi_\psi)))
         & \cong \bigoplus_{g\in K_m \backslash \widetilde{G}/ \widetilde{B}} \Hom_{K_m \cap \prescript{g}{}{\widetilde{B}}}(\eta, \prescript{g}{}{(|-| \mu \chi_\psi)}) \\
         & \cong \bigoplus_{g\in K_m \backslash G/ B} \Hom_{K_m \cap \prescript{g}{}{B}}(\eta, \prescript{g}{}{\mu})                                                     \\
         & \cong \bigoplus_{g\in K_m \backslash G/ B} \Hom_{K_m^g \cap B}(\eta^g, \mu),
    \end{align*}
    where $\eta^g$ and $\mu$ denote characters which send $\left( \begin{array}{cc}a&b\\ c&d \end{array} \right)$ to $\eta(gdg^{-1})$ and $\mu(a)$, respectively.
    Thus we have
    \begin{align*}
        \dim_\C \pi_\psi(\mu)^{K_m}_\eta
         & =\sum_{g\in K_m \backslash G/ B} \dim_\C \Hom_{K_m^g \cap B}(\eta^g, \mu),
    \end{align*}
    which is equal to the number of double cosets $K_m g B$ such that $\eta^g=\mu$ on $K_m^g\cap B$.
    Since we are now considering the case $\varepsilon=0$, the double cosets are classified in Lemma \ref{lemcoset}.
    If $g=1_2$ (resp. $g=w$), then it can be easily seen that $\eta^g=\mu$ on $K_m^g\cap B$ if and only if $\eta=\mu^{-1}$ (resp. $\eta=\mu$).
    Now we assume that $m\geq2$ and consider the case $g=n^{\mathrm{op}}(z_i)$, where $1\leq i \leq m-1$ and $z_i=\varpi^i$ or $\xi\varpi^i$.
    A straightforward calculation shows that
    \begin{align*}
        \left( \left( \begin{array}{cc}a&b\\ c&d \end{array} \right)^g \right)
        =\left( \begin{array}{cc}1&\\ -z_i&1 \end{array} \right) \left( \begin{array}{cc}a&b\\ c&d \end{array} \right) \left( \begin{array}{cc}1&\\ z_i&1 \end{array} \right)
        = \left( \begin{array}{cc}a+z_ib&b\\ c-z_a+z_id-z_i^2b&d-z_ib \end{array} \right),
    \end{align*}
    and hence
    \begin{align*}
        K_m^g \cap B
         & =\Set{\left( \begin{array}{cc}a+z_ib&b\\ 0&d-z_ib \end{array} \right) | a,d\in\mathcal{O}^\times,\ b\in\mathcal{O},\ c\in\mathfrak{p}^m,\ ad-bc=1,\ c-z_ia+z_id-z_i^2b=0 }       \\
         & =\Set{\left( \begin{array}{cc}(d-z_ib)^{-1}&b\\ 0&d-z_ib \end{array} \right) | d\in\mathcal{O}^\times,\ b\in\mathcal{O},\ z_i((d-z_ib)^{-1}-z_ib)-z_id+z_i^2b\in\mathfrak{p}^m } \\
         & =\Set{\left( \begin{array}{cc}(d-z_ib)^{-1}&b\\ 0&d-z_ib \end{array} \right) | d\in\mathcal{O}^\times,\ b\in\mathcal{O},\ d^2-1-z_ibd\in\mathfrak{p}^{m-i} }.
    \end{align*}
    Thus, for
    \begin{align*}
        x=\left( \begin{array}{cc}(d-z_ib)^{-1}&b\\ 0&d-z_ib \end{array} \right)\in K_m^g \cap B,
    \end{align*}
    we have $\eta^g(x)=\eta(d)$ and $\mu(x)=\mu(d-z_ib)^{-1}$.
    In the case $i\geq m/2$, since $i\geq m-i$, the condition
    \begin{equation*}
        d^2-1-z_ibd\in\mathfrak{p}^{m-i}
    \end{equation*}
    is equivalent to
    \begin{equation*}
        d\in \pm1+\mathfrak{p}^{m-i}.
    \end{equation*}
    Then one can see that $\eta^g=\mu$ on $K_m^g\cap B$ if and only if $c(\mu)\leq i$ and $c(\eta\mu)\leq m-i$.
    In the case $i\leq m/2$, put $c'=d^2-1-z_ibd$.
    Since $i\leq m-i$, the set
    \begin{align*}
        \set{(d,b)\in\mathcal{O}^\times \times \mathcal{O} | d^2-1-z_ibd\in\mathfrak{p}^{m-i}}
    \end{align*}
    is equal to
    \begin{align*}
        \Set{\left( d, \frac{d^2-1-c'}{z_id} \right)| d\in\pm1+\mathfrak{p}^i, c'\in \mathfrak{p}^{m-i}}.
    \end{align*}
    Then one can see that $\eta^g=\mu$ on $K_m^g\cap B$ if and only if $c(\mu)\leq m-i$ and $c(\eta\mu^{-1})\leq i$.
    Now the assertions for $\varepsilon=0$ follow.

    We finally consider the case $\varepsilon=1$.
    As in the case of $\varepsilon=0$, we can see that the dimension of $\pi_\psi(\mu)^{K^1_m}_\eta$ is equal to the number of double cosets $K^1_mgB$ such that $\eta^g=\mu$ on $(K^1_m)^g\cap B$.
    Note that $\eta^\beta=\eta$ and $\mu^\beta=\mu$ as functions on $K^1$.
    Hence we can see that $\eta^g=\mu$ on $(K^1_m)^g\cap B$ if and only if $\eta^{\beta^{-1} g \beta}=\mu$ on $(K^0_m)^{\beta^{-1} g\beta}\cap B$.
    Since a bijection $g\mapsto \beta^{-1}g\beta$ from $G$ to $G$ gives a bijection between $K^1_m\backslash G/ B$ and $K^0_m\backslash G/ B$, we are reduced to the case $\varepsilon=0$.
\end{proof}
\begin{corollary}\label{corps}
    Let $\mu$ be a character of $F^\times$ and $\eta$ that of $\mathcal{O}^\times$ with $\eta(-1)=\mu(-1)$.
    Put $M=c^\varepsilon_\eta(\pi_\psi(\mu))$.
    \begin{enumerate}[(1)]
        \item If $c(\mu)=c(\eta)=0$, then we have $M=0$ and
              \begin{align*}
                  \dim_\C \pi_\psi(\mu)^{K_0, \new}_1
                                                            & =\dim_\C \pi_\psi(\mu)^{K_1, \new}_1
                  =\dim_\C \pi_\psi(\mu)^{K_2, \new}_1
                  =\dim_\C  \pi_\psi(\mu)^{K_3, \new}_1 =1, &                                      &  &                     \\
                  \dim_\C  \pi_\psi(\mu)^{K_m, \new}_1      & = 0,                                 &  & \text{if }m\geq4. &
              \end{align*}
        \item If $\eta=\mu|_{\mathcal{O}^\times}\neq \mu^{-1}|_{\mathcal{O}^\times}$ or $\eta=\mu^{-1}|_{\mathcal{O}^\times}\neq \mu|_{\mathcal{O}^\times}$, then we have $M=c(\mu)$ and
              \begin{align*}
                  \dim_\C  \pi_\psi(\mu)^{K_M, \new}_1     & = 1, &  &                       & \\
                  \dim_\C  \pi_\psi(\mu)^{K_{M+1}, \new}_1 & = 2, &  &                       & \\
                  \dim_\C  \pi_\psi(\mu)^{K_{M+2}, \new}_1 & = 1, &  &                       & \\
                  \dim_\C  \pi_\psi(\mu)^{K_m, \new}_1     & = 0, &  & \text{if } m\geq M+3. &
              \end{align*}
        \item Otherwise, we have
              \begin{align*}
                  \dim_\C  \pi_\psi(\mu)^{K_M, \new}_1     & = 2, &  &                       & \\
                  \dim_\C  \pi_\psi(\mu)^{K_{M+1}, \new}_1 & = 2, &  &                       & \\
                  \dim_\C  \pi_\psi(\mu)^{K_m, \new}_1     & = 0, &  & \text{if } m\geq M+2. &
              \end{align*}
    \end{enumerate}
\end{corollary}
\begin{proof}
    The assertions (2) and (3) follow easily from Theorems \ref{thm:rs-Mp2} and \ref{thmps} and Proposition \ref{prop:numgeneric}.
    Assume that both of $\mu$ and $\eta$ are unramified.
    Suppose that $\alpha_2 (\pi_\psi(\mu)^{K_0}_1)$ is not contained in $\pi_\psi(\mu)^{K_1}_1$.
    Then, by Theorem \ref{thmps}, we have
    \begin{align*}
        \dim_\C  \pi_\psi(\mu)^{K_0, \new}_1 & = \dim_\C  \pi_\psi(\mu)^{K_0}_1 =1,                                                                          \\
        \dim_\C  \pi_\psi(\mu)^{K_1, \new}_1 & = \dim_\C  \pi_\psi(\mu)^{K_1}_1 - \dim_\C  \pi_\psi(\mu)^{K_0}_1= 2-1 =1,                                    \\
        \dim_\C  \pi_\psi(\mu)^{K_2, \new}_1 & = \dim_\C  \pi_\psi(\mu)^{K_2}_1 - \dim_\C  \pi_\psi(\mu)^{K_1}_1 - \dim_\C  \pi_\psi(\mu)^{K_0}_1 = 4-2-1=1,
    \end{align*}
    and
    \begin{align*}
        \dim_\C  \pi_\psi(\mu)^{K_3, \new}_1
         & = \dim_\C  \pi_\psi(\mu)^{K_3}_1 - \dim_\C  \pi_\psi(\mu)^{K_2}_1 - \dim_\C  \pi_\psi(\mu)^{K_1}_1 + \dim_\C \pi_\psi(\mu)^{K_2}_1 \cap \alpha_2 \pi_\psi(\mu)^{K_1}_1 \\
         & \geq 6-4-2 +1= 1.
    \end{align*}
    Combining this with Theorem \ref{thm:rs-Mp2}, we have
    \begin{align*}
        \dim_\C  \pi_\psi(\mu)^{K_3, \new}_1 & = 1, \\
        \dim_\C  \pi_\psi(\mu)^{K_m, \new}_1 & = 0,
    \end{align*}
    for $m\geq 4$.
    It is thus enough to show that $\alpha_2 (\pi_\psi(\mu)^{K_0}_1)$ is not contained in $\pi_\psi(\mu)^{K_1}_1$.
    Let $f^0$ be a nonzero function in $\pi_\psi(\mu)^{K_0}_1$.
    We shall show that $\alpha_2f^0$ does not belong to $\pi_\psi(\mu)^{K_1}_1$.
    Let $c\in \varpi^{1+\varepsilon}\mathcal{O}^\times$.
    Since
    \begin{equation*}
        (n^{\mathrm{op}}(c), \bm{s}^\varepsilon(n^{\mathrm{op}}(c)) ) (t(\varpi^{-1}), 1) = (\left( \begin{array}{cc}\varpi^{-1}&\\ c\varpi^{-1}&\varpi \end{array} \right), (\varpi,c)_2),
    \end{equation*}
    we have
    \begin{equation*}
        [n^{\mathrm{op}}(c) \cdot \alpha_2 f^0](1) = f^0((\left( \begin{array}{cc}\varpi^{-1}&\\ c\varpi^{-1}&\varpi \end{array} \right), (\varpi,c)_2)).
    \end{equation*}
    On the other hand, we have
    \begin{align*}
         & (n(c^{-1}), 1) \cdot (wt(\varpi^\varepsilon), \bm{s}^\varepsilon(wt(\varpi^\varepsilon))) \cdot (n(c\varpi^{-2\varepsilon}), \bm{s}^\varepsilon(n(c\varpi^{-2\varepsilon}))) \cdot (t(-c\varpi^{-1-\varepsilon}), \bm{s}^\varepsilon(t(-c\varpi^{-1-\varepsilon})))                                    \\
         & = ( \left( \begin{array}{cc}-c^{-1}\varpi^\varepsilon&\varpi^{-\varepsilon}\\ -\varpi^\varepsilon& \end{array} \right), 1) (\left( \begin{array}{cc}-c\varpi^{-1-\varepsilon}&-\varpi^{1-\varepsilon}\\ &-c^{-1}\varpi^{1+\varepsilon} \end{array} \right), \gamma_F(-c\varpi^{-1-\varepsilon}, \psi)) \\
         & = (\left( \begin{array}{cc}\varpi^{-1}&\\ c\varpi^{-1}&\varpi \end{array} \right), (-c\varpi^{-1+\varepsilon}, -\varpi^\varepsilon)_2 \gamma_F(-c\varpi^{-1-\varepsilon}, \psi)).
    \end{align*}
    Since $f^0$ is left $N$-invariant and right $K^\varepsilon_0$-invariant, this implies that
    \begin{equation*}
        f^0((\left( \begin{array}{cc}\varpi^{-1}&\\ c\varpi^{-1}&\varpi \end{array} \right), (\varpi,c)_2)) = (\varpi,c)_2 (-c\varpi^{-1+\varepsilon}, -\varpi^\varepsilon)_2 \gamma_F(-c\varpi^{-1-\varepsilon}, \psi) f^0(1).
    \end{equation*}
    Moreover, $(-c\varpi^{-1+\varepsilon}, -\varpi^\varepsilon)_2 \gamma_F(-c\varpi^{-1-\varepsilon}, \psi)$ is equal to
    \begin{align*}
         & (-c\varpi, -\varpi^\varepsilon)_2 \gamma_F(-c\varpi^{1+\varepsilon}, \psi)     \\
         & = (-c\varpi, -1)_2 \gamma_F(-c\varpi, \psi) \gamma_F(\varpi^\varepsilon, \psi) \\
         & = (-c\varpi, -1)_2 \gamma_F(\varpi^\varepsilon, \psi)^2                        \\
         & = (-c\varpi, -1)_2 (\varpi^\varepsilon, -1)_2                                  \\
         & = (-c\varpi^{-1-\varepsilon}, -1)_2 = 1.
    \end{align*}
    Therefore, we have
    \begin{align*}
        [n^{\mathrm{op}}(c) \cdot \alpha_2 f^0](1) = (\varpi,c)_2 f^0(1).
    \end{align*}
    Since $n^{\mathrm{op}}(c)$ is an element in $K_1$, this means that $\alpha_2 f^0$ is not an element in $\pi_\psi(\mu)^{K_1}_1$.
    This completes the proof.
\end{proof}

Although the dimension of a space $\pi_\psi(\mu)^{K_{c^\varepsilon_\eta(\pi_\psi(\mu))}}_\eta$ of local newforms is greater than one for some $\eta$, in this paper, we will give a multiplicity one property for local newforms in terms of the Whittaker functional for certain important characters $\eta$.
For a nontrivial additive character $\Psi$ of $F$, we consider a $\Psi$-Whittaker functional $\lambda_\Psi$ on $\pi_\psi(\mu)$ defined by
\begin{equation*}
    \lambda_\Psi(f) = \lim_{r\to \infty} \int_{\mathfrak{p}^{-r}} f((w n(x), 1)) \overline{\Psi(x)} dx, \qquad f \in \pi_\psi(\mu).
\end{equation*}
\begin{corollary}\label{corwhittakerps}
    Let $\mu$ be a character of $F^\times$, and $\eta$ a character of $\mathcal{O}^\times$ such that $\eta(-1)=\mu(-1)$.
    Put $M=c^\varepsilon_\eta(\pi_\psi(\mu))$.
    For $m\geq M$ and $i \geq 0$, let $f^m_w$ (resp. $f^m_1$, resp. $f^m_{i,2}$, resp. $f^m_{i, \xi}$) be a nonzero vector in $\pi_\psi(\mu)^{K_m}_\eta$ supported on $\widetilde{B}(x,1)K_m$, where $x=w$ (resp. $1$, resp. $n^{\mathrm{op}}(\varpi^{c(\eta\mu^{-1})+i+\varepsilon})$, resp. $n^{\mathrm{op}}(\xi\varpi^{c(\eta\mu^{-1})+i+\varepsilon})$), if it exists.
    \begin{enumerate}[(1)]
        \item Assume that $c(\mu)=c(\eta)=0$. Then we have $M=0$ and
              \begin{align*}
                  \pi_\psi(\mu)^{K_0}_1 & =\C f^0_w,                                                           \\
                  \pi_\psi(\mu)^{K_1}_1 & =\C f^1_w \oplus \C f^1_1,                                           \\
                  \pi_\psi(\mu)^{K_2}_1 & =\C f^2_w \oplus \C f^2_1 \oplus \C f^2_{1,2} \oplus \C f^2_{1,\xi}.
              \end{align*}
              We also have
              \begin{itemize}
                  \item $\lambda_\psi(f^0_w) \neq 0$ (resp. $\lambda_{\psi_\xi}(f^0_w) \neq 0$) unless $\mu = \chi_\xi \cdot |-|^{-\frac{1}{2}}$ (resp. $\mu = |-|^{-\frac{1}{2}}$);
                  \item $\lambda_\psi(f^1_w)\neq 0$, $\lambda_\psi(f^1_1)\neq 0$, $\lambda_{\psi_\xi}(f^1_w)\neq 0$, and $\lambda_{\psi_\xi}(f^1_1)\neq 0$;
                  \item at least one of $\lambda_{\psi_\varpi}(f^2_{1,2})$ (resp. $\lambda_{\psi_\varpi}(f^2_{1,2})$, resp. $\lambda_{\psi_\varpi}(f^2_{1,\xi})$, resp. $\lambda_{\psi_{\xi\varpi}}(f^2_{1,2})$) or $\lambda_{\psi_\varpi}(f^2_{1,\xi})$ (resp. $\lambda_{\psi_{\xi\varpi}}(f^2_{1,2})$, resp. $\lambda_{\psi_{\xi\varpi}}(f^2_{1,\xi})$, resp. $\lambda_{\psi_{\xi\varpi}}(f^2_{1,\xi})$) is nonzero, for any $\mu$.
              \end{itemize}
        \item Assume that $\mu$ is nontrivial quadratic on $\mathcal{O}^\times$ and $\eta=\mu|_{\mathcal{O}^\times}$.
              Then we have $M=1$ and
              \begin{align*}
                  \pi_\psi(\mu)^{K_1}_\eta & =\C f^1_w \oplus \C f^1_1,                                           \\
                  \pi_\psi(\mu)^{K_2}_\eta & =\C f^2_w \oplus \C f^1_1 \oplus \C f^2_{1,2} \oplus \C f^2_{1,\xi}.
              \end{align*}
              We also have
              \begin{itemize}
                  \item $\lambda_\psi(f^1_w)\neq 0$, $\lambda_\psi(f^1_1)\neq 0$, $\lambda_{\psi_\xi}(f^1_w)\neq 0$, and $\lambda_{\psi_\xi}(f^1_1)\neq 0$;
                  \item at least one of $\lambda_{\psi_\varpi}(f^2_{1,2})$ (resp. $\lambda_{\psi_\varpi}(f^2_{1,2})$, resp. $\lambda_{\psi_\varpi}(f^2_{1,\xi})$, resp. $\lambda_{\psi_{\xi\varpi}}(f^2_{1,2})$) or $\lambda_{\psi_\varpi}(f^2_{1,\xi})$ (resp. $\lambda_{\psi_{\xi\varpi}}(f^2_{1,2})$, resp. $\lambda_{\psi_{\xi\varpi}}(f^2_{1,\xi})$, resp. $\lambda_{\psi_{\xi\varpi}}(f^2_{1,\xi})$) is nonzero.
              \end{itemize}
        \item Assume that $\eta=\mu|_{\mathcal{O}^\times}\neq \mu^{-1}|_{\mathcal{O}^\times}$.
              Then we have $M=c(\mu)$ and
              \begin{align*}
                  \pi_\psi(\mu)^{K_M}_\eta     & =\C f^M_w,                                                       \\
                  \pi_\psi(\mu)^{K_{M+1}}_\eta & =\C f^{M+1}_w \oplus \C f^{M+1}_{1,2} \oplus \C f^{M+1}_{1,\xi}.
              \end{align*}
              We also have
              \begin{itemize}
                  \item $\lambda_\psi(f^M_w)\neq 0$ and $\lambda_{\psi_\xi}(f^M_w)\neq 0$;
                  \item at least one of $\lambda_{\psi_\varpi}(f^{M+1}_{1,2})$ (resp. $\lambda_{\psi_\varpi}(f^{M+1}_{1,2})$, resp. $\lambda_{\psi_\varpi}(f^{M+1}_{1,\xi})$, resp. $\lambda_{\psi_{\xi\varpi}}(f^{M+1}_{1,2})$) or $\lambda_{\psi_\varpi}(f^{M+1}_{1,\xi})$ (resp. $\lambda_{\psi_{\xi\varpi}}(f^{M+1}_{1,2})$, resp. $\lambda_{\psi_{\xi\varpi}}(f^{M+1}_{1,\xi})$, resp. $\lambda_{\psi_{\xi\varpi}}(f^{M+1}_{1,\xi})$) is nonzero.
              \end{itemize}
        \item Assume that $0<c(\eta\mu^{-1})\leq c(\eta\mu)$.
              Then we have
              \begin{align*}
                  \pi_\psi(\mu)^{K_M}_\eta     & =\C f^M_{0,2} \oplus \C f^M_{0,\xi},                                                           \\
                  \pi_\psi(\mu)^{K_{M+1}}_\eta & =\C f^{M+1}_{0,2} \oplus \C f^{M+1}_{0,\xi} \oplus \C f^{M+1}_{1,2} \oplus \C f^{M+1}_{1,\xi}.
              \end{align*}
              Moreover, for each $i=0, 1$, at least one of $\lambda_{\psi_{\varpi^i}}(f^{M+i}_{i,2})$ (resp. $\lambda_{\psi_{\varpi^i}}(f^{M+i}_{i,2})$, resp. $\lambda_{\psi_{\varpi^i}}(f^{M+i}_{i,\xi})$, resp. $\lambda_{\psi_{\xi\varpi^i}}(f^{M+i}_{i,2})$) or $\lambda_{\psi_{\varpi^i}}(f^{M+i}_{i,\xi})$ (resp. $\lambda_{\psi_{\xi\varpi^i}}(f^{M+i}_{i,2})$, resp. $\lambda_{\psi_{\xi\varpi^i}}(f^{M+i}_{i,\xi})$, resp. $\lambda_{\psi_{\xi\varpi^i}}(f^{M+i}_{i,\xi})$) is nonzero.
    \end{enumerate}
\end{corollary}
\begin{proof}
    The assertions on the structures of $\pi_\psi(\mu)^{K_m}_\eta$ follow from Theorem \ref{thmps}.
    We shall now calculate the Whittaker functionals.

    Suppose first that both $\eta$ and $\mu$ are unramified, and consider $f^0_w$.
    For $\Psi=\psi$ or $\psi_\xi$, we have
    \begin{align*}
        \lambda_\Psi(f^0_w)
         & =\lim_{r\to \infty} \int_{\mathfrak{p}^{-r}} f^0_w((w n(x), 1)) \overline{\Psi(x)} dx                                                                                                                                                                                               \\
         & =\int_{\mathfrak{p}^{-\varepsilon}} f^0_w((w n(x), 1)) \overline{\Psi(x)} dx + \sum_{r=\varepsilon+1}^{\infty} \int_{\varpi^{-r}\mathcal{O}^\times} f^0_w((w n(x), 1)) \overline{\Psi(x)} dx                                                                                        \\
         & = f^0_w((w,1)) \int_{\mathfrak{p}^{-\varepsilon}} dx + \sum_{r=\varepsilon+1}^{\infty} \int_{\varpi^{-r}\mathcal{O}^\times} f^0_w((\left( \begin{array}{cc}-x^{-1}&1\\ &-x \end{array} \right) \left( \begin{array}{cc}1&\\ x^{-1}&1 \end{array} \right), 1)) \overline{\Psi(x)} dx \\
         & = q^\varepsilon f^0_w((w,1)) + f^0_w(1) \sum_{r=\varepsilon+1}^{\infty} \mu(\varpi)^r \int_{\mathcal{O}^\times} \gamma_F(\varpi^{-r} u, \psi)^{-1} \Psi(\varpi^{-r}u) du,
    \end{align*}
    where we change variables $x=-\varpi^{-r}u$.
    By Lemma 1.11 in \cite{szp}, the last integral vanishes for $r>\varepsilon+1$.
    Hence we have
    \begin{align*}
        \lambda_\Psi(f^0_w)
         & = q^\varepsilon f^0_w((w,1))  +  \mu(\varpi)^{\varepsilon+1} f^0_w(1) \int_{\mathcal{O}^\times} \gamma_F(\varpi^{-\varepsilon-1} u, \psi'_{\varpi^\varepsilon})^{-1} \Psi'(\varpi^{-1} u) du.
    \end{align*}
    where $\psi'=\psi_{\varpi^{-\varepsilon}}$ and $\Psi'=\Psi_{\varpi^{-\varepsilon}}$.
    Since
    \begin{equation*}
        (n(\varpi^{-\varepsilon}),1) (n^{\mathrm{op}}(-\varpi^\varepsilon),1) (n(\varpi^{-\varepsilon}), 1) = (t(\varpi^{-\varepsilon}), (-1,\varpi)_2^\varepsilon ) (w,1),
    \end{equation*}
    and the left hand side is an element of $K_0$, we have
    \begin{equation*}
        f^0_w(1)=q^\varepsilon \mu(\varpi)^{-\varepsilon} (-1,\varpi)_2^\varepsilon \gamma_F(\varpi^\varepsilon,\psi)^{-1} f^0_w((w,1)).
    \end{equation*}
    Thus we have
    \begin{align*}
        \frac{1}{q^\varepsilon f^0_w((w,1))} \lambda_\psi(f^0_w)
         & = 1 + \mu(\varpi) (-1,\varpi)_2^\varepsilon \gamma_F(\varpi^\varepsilon,\psi)^{-1} \int_{\mathcal{O}^\times} \gamma_F(\varpi^{-\varepsilon-1} u, \psi'_{\varpi^\varepsilon})^{-1} \Psi'(\varpi^{-1} u) du \\
         & = 1 + \mu(\varpi) \int_{\mathcal{O}^\times} \gamma_F(\varpi^{-1} u, \psi')^{-1} \Psi'(\varpi^{-1} u) du.
    \end{align*}
    Note that the conductor of $\psi'$ is 0.
    If $\Psi=\psi$, then by Lemma 1.12 in \cite{szp}, the last integral is equal to $q^{-\frac{1}{2}}$.
    Therefore we have
    \begin{equation*}
        \frac{1}{q^\varepsilon f^0_w((w,1))} \lambda_\psi(f^0_w)
        = 1 + q^{-\frac{1}{2}} \mu(\varpi),
    \end{equation*}
    which means that $\lambda_\psi(f^0_w) \neq 0$ if and only if $\mu \neq \chi_\xi \cdot |-|^{-\frac{1}{2}}$.
    On the other hand, if $\Psi=\psi_\xi$, the last integral is equal to $\chi_\xi(\varpi) q^{-\frac{1}{2}}$, since $\gamma_F(\varpi^{-1} \xi u, \psi')^{-1}=(\varpi, \xi)_2 \gamma_F(\varpi^{-1} u, \psi')^{-1}$.
    Thus we have $\lambda_{\psi_\xi}(f^0_w) \neq 0$ if and only if $\mu \neq |-|^{-\frac{1}{2}}$.

    Suppose next that $\mu^2|_{\mathcal{O}^\times}$ is trivial and $\eta=\mu|_{\mathcal{O}^\times}$, and we consider $f^1_1 \in \pi_\psi(\mu)^{K_1}_\eta$.
    For $\Psi=\psi$ or $\psi_\xi$, we have
    \begin{align*}
        \lambda_\Psi(f^1_1)
         & = \int_{\mathfrak{p}^{-\varepsilon}} f^1_1((w n(x), 1)) \overline{\Psi(x)} dx + \sum_{r=\varepsilon+1}^{\infty} \int_{\varpi^{-r}\mathcal{O}^\times} f^1_1((\left( \begin{array}{cc}-x^{-1}&1\\ &-x \end{array} \right) \left( \begin{array}{cc}1&\\ x^{-1}&1 \end{array} \right), 1)) \overline{\Psi(x)} dx \\
         & = 0 + \sum_{r=\varepsilon+1}^{\infty} \int_{\varpi^{-r}\mathcal{O}^\times} |x|^{-1} \mu(-x)^{-1} \gamma_F(-x^{-1}, \psi)^{-1} f^1_1(1) \overline{\Psi(x)} dx.
    \end{align*}
    Changing the variables $x\mapsto u=-\varpi^r x$, we see that this equals
    \begin{align*}
         & f^1_1(1) \sum_{r=\varepsilon+1}^{\infty} \mu(\varpi)^r \int_{\mathcal{O}^\times} \mu(u)^{-1} \gamma_F(\varpi^r u^{-1}, \psi)^{-1} \overline{\Psi(-\varpi^{-r}u)} du                                                                                                        \\
         & = f^1_1(1) \gamma_F(\varpi^{-\varepsilon},\psi)^{-1} \sum_{r=\varepsilon+1}^{\infty} \mu(\varpi)^r \int_{\mathcal{O}^\times} \mu(u)^{-1} \gamma_F(\varpi^{-r+\varepsilon}u , \psi_{\varpi^{-\varepsilon}})^{-1} \Psi_{\varpi^{-\varepsilon}}(\varpi^{-r+\varepsilon}u) du.
    \end{align*}
    Since $\gamma_F(\varpi^{-r+\varepsilon}u , \psi_{\varpi^{-\varepsilon}})=(\varpi^{-r+\varepsilon},\xi)_2 \gamma_F(\varpi^{-r+\varepsilon}u , \psi_{\xi\varpi^{-\varepsilon}})$, it follows from the proofs of \cite[Lemmas 1.11, 1.12, and 1.13]{szp} that the integral
    \begin{equation*}
        \int_{\mathcal{O}^\times} \mu(u)^{-1} \gamma_F(\varpi^{-r+\varepsilon}u , \psi_{\varpi^{-\varepsilon}})^{-1} \Psi_{\varpi^{-\varepsilon}}(\varpi^{-r+\varepsilon}u) du
    \end{equation*}
    is nonzero if and only if $r=\varepsilon+1$, and hence we have $\lambda_\Psi(f^1_1) \neq 0$.

    Third, suppose that $m\geq 1$ and $f^m_w \in \pi_\psi(\mu)^{K_m}_\eta$.
    For $\Psi=\psi$ or $\psi_\xi$, by Lemma \ref{lemcoset} and Corollary \ref{corcoset}, we have
    \begin{align*}
        \lambda_\Psi(f^m_w)
         & =\int_{\mathfrak{p}^{-\varepsilon}} f^m_w((w n(x), 1)) \overline{\Psi(x)} dx + \sum_{r=\varepsilon+1}^{\infty} \int_{\varpi^{-r}\mathcal{O}^\times} f^m_w((w n(x), 1)) \overline{\Psi(x)} dx                                                                                        \\
         & = f^m_w((w,1)) \int_{\mathfrak{p}^{-\varepsilon}} dx + \sum_{r=\varepsilon+1}^{\infty} \int_{\varpi^{-r}\mathcal{O}^\times} f^m_w((\left( \begin{array}{cc}-x^{-1}&1\\ &-x \end{array} \right) \left( \begin{array}{cc}1&\\ x^{-1}&1 \end{array} \right), 1)) \overline{\Psi(x)} dx \\
         & = q^\varepsilon f^m_w((w,1)) + 0 \neq 0.
    \end{align*}

    Finally we consider $f^m_{i,2}$ and $f^m_{i,\xi}$.
    Let $i=0$ or $1$, and $m \geq M$.
    Suppose that $\pi_\psi(\mu)^{K_m}_\eta$ contains $f^m_{i,2}$ and $f^m_{i,\xi}$.
    Since $f^m_{i,2}$ is supported on $\widetilde{B}(n^{\mathrm{op}}(\varpi^{c(\eta\mu^{-1})+i+\varepsilon}),1)K_m$, the routine calculation shows that $\lambda_{\psi_{\varpi^i}}(f^m_{i,2})$ equals
    \begin{equation*}
        \int_{\varpi^{-c(\eta\mu^{-1})-i-\varepsilon}\mathcal{O}^{\times2}} f^m_{i,2}((\left( \begin{array}{cc}-x^{-1}&1\\ &-x \end{array} \right) \left( \begin{array}{cc}1&\\ x^{-1}&1 \end{array} \right), 1)) \overline{\psi_{\varpi^i}(x)} dx.
    \end{equation*}
    Put $x=\varpi^{-c(\eta\mu^{-1})-i-\varepsilon} u^2$ and change the variables $x\mapsto u$.
    Then, up to a nonzero constant multiple, this equals
    \begin{align*}
         & \int_{\mathcal{O}^\times} \mu(u)^{-2}  f^m_{i,2}(( n^{\mathrm{op}}(\varpi^{c(\eta\mu^{-1})+i+\varepsilon} u^{-2}), 1)) \psi(-\varpi^{-c(\eta\mu^{-1})-\varepsilon} u^2) du                         \\
         & =\int_{\mathcal{O}^\times} \mu(u)^{-2}  f^m_{i,2}((t(u), (u,u)_2) (n^{\mathrm{op}}(\varpi^{c(\eta\mu^{-1})+i+\varepsilon}), 1) (t(u^{-1}), 1)) \psi(-\varpi^{-c(\eta\mu^{-1})-\varepsilon} u^2) du \\
         & =\int_{\mathcal{O}^\times} \mu(u)^{-2} \cdot \mu(u) (u,u)_2 \gamma_F(u,\psi)^{-1}                                                                                                                  \\
         & \qquad \times \gamma_F(u^{-1},\psi)^{-1}\eta(u) f^m_{i,2}( (n^{\mathrm{op}}(\varpi^{c(\eta\mu^{-1})+i+\varepsilon}), 1) ) \psi(-\varpi^{-c(\eta\mu^{-1})-\varepsilon} u^2) du                      \\
         & =f^m_{i,2}( (n^{\mathrm{op}}(\varpi^{c(\eta\mu^{-1})+i+\varepsilon}), 1) ) h(\eta\mu^{-1}, \psi_{-\varpi^{-c(\eta\mu^{-1})-\varepsilon}}).
    \end{align*}
    Therefore, $\lambda_\psi(f^m_{i,2})$ is nonzero if and only if $h(\eta\mu^{-1}, \psi_{-\varpi^{-c(\eta\mu^{-1})-\varepsilon}})$ is.
    Similarly, $\lambda_\psi(f^m_{i,\xi})$ (resp. $\lambda_{\psi_\xi}(f^m_{i,2})$, resp. $\lambda_{\psi_\xi}(f^m_{i,\xi})$) is nonzero if and only if $h(\eta\mu^{-1}, \psi_{-\xi\varpi^{-c(\eta\mu^{-1})-\varepsilon}})$ (resp. $h(\eta\mu^{-1}, \psi_{-\xi\varpi^{-c(\eta\mu^{-1})-\varepsilon}})$, resp. $h(\eta\mu^{-1}, \psi_{-\varpi^{-c(\eta\mu^{-1})-\varepsilon}})$) is.
    Now the assertions follow from Lemma \ref{lemgausssum2}.
\end{proof}
\subsection{Even Weil representations and Steinberg representations}\label{subsec:evenweilsteinberg}
Next we shall consider the even Weil representations and the Steinberg representations.
\begin{theorem}\label{thmevenweil}
    Let $\chi$ be a quadratic or trivial character of $F^\times$, and $\eta$ a character of $\mathcal{O}^\times$ such that $\eta(-1)=\chi(-1)$.
    Then we have $c^\varepsilon_\eta(\omega_{\psi,\chi}^+)=2c(\eta\chi)+c(\chi)$ and
    \begin{align*}
        \dim_\C \left( \omega_{\psi,\chi}^+ \right)^{K_m}_\eta
        =\left( \left\lfloor \frac{m-2c(\eta\chi)-c(\chi)}{2} \right\rfloor +1 \right)^+.
    \end{align*}
    In particular, $c^\varepsilon_{\min}(\omega_{\psi,\chi}^+)=c^\varepsilon_\eta(\omega_{\psi,\chi}^+)$ if and only if $\eta=\chi|_{\mathcal{O}^\times}$, and $c^\varepsilon_{\min}(\omega_{\psi,\chi}^+)=c(\chi)$.
\end{theorem}
\begin{proof}
    Let $\psi'$ be an additive character of $F$ such that $\omega_{\psi,\chi}^+\cong \omega_{\psi'}^+$.
    Put $\nu=c(\psi')+\varepsilon$, which is even (resp. odd) if $\chi$ is unramified (resp. ramified).
    Let us realize the even Weil representation $\omega_{\psi,\chi}^+$ as $(\omega_{\psi'}^+, \mathcal{S}^+(F))$, and $\varphi$ be any nonzero vector in $(\omega_{\psi'}^+, \mathcal{S}^+(F))$.
    Since $K_m$ is generated by $t(a)$, $n(b)$, and $n^{\mathrm{op}}(c)$, where $a\in\mathcal{O}^\times$, $b\in\mathfrak{p}^{-\varepsilon}$, and $c\in\mathfrak{p}^{m+\varepsilon}$, the vector $\varphi$ is contained in $(\omega_{\psi'}^+)^{K_m}_\eta$ if and only if all of the followings hold:
    \begin{enumerate}[(a)]
        \item $\omega_{\psi'}^+(t(a))\varphi=\eta(a^{-1})\varphi$, for all $a\in\mathcal{O}^\times$;
        \item $\omega_{\psi'}^+(n(b))\varphi=\varphi$, for all $b\in\mathfrak{p}^{-\varepsilon}$;
        \item $\omega_{\psi'}^+(n^{\mathrm{op}}(c))\varphi=\varphi$, for all $c\in\mathfrak{p}^{m+\varepsilon}$.
    \end{enumerate}

    Suppose that the conditions (a)-(c) hold.
    We shall determine the function $\varphi$.
    First we consider the condition (a).
    Since $\omega_{\psi'}^+(t(a))\varphi(y)$ equals to $\gamma_F(a,\psi)\gamma_F(a,\psi')^{-1}\varphi(ay)$ for any $a\in \mathcal{O}^\times$ and $y\in F$, the condition (a) is equivalent to
    \begin{align}\label{eqa}
        \varphi(ay)=\gamma_F(a,\psi')\gamma_F(a,\psi)^{-1} \eta(a)^{-1}\varphi(y)=\chi\eta^{-1}(a)\varphi(y),\quad \text{for all $a\in \mathcal{O}^\times$ and $y\in F$.}
    \end{align}
    Therefore, $\varphi$ is determined by $\{\varphi(\varpi^i) \mid i\in\Z\}$.
    Note that $\varphi(\varpi^i)$ is constant for sufficiently large $i$, since $\varphi$ is locally constant.

    Next we shall consider the condition (b).
    We know that $\omega_{\psi'}^+(n(b))\varphi(y)$ equals to $\psi'(by^2)\varphi(y)$ for any $b\in\mathfrak{p}^{-\varepsilon}$ and $y\in F$.
    Suppose that (b) holds.
    Then for any $y\in \operatorname{supp}(\varphi)$, we have $\psi'(by^2)=1$ for all $b\in\mathfrak{p}^{-\varepsilon}$.
    This implies that $y^2\in\mathfrak{p}^\nu$.
    Hence we have
    \begin{equation}\label{eqb}
        \operatorname{supp}(\varphi) \subset \mathfrak{p}^{\left\lceil \frac{\nu}{2} \right\rceil}.
    \end{equation}

    Now we come to the condition (c).
    Since $(w,1)n^{\mathrm{op}}(c)=n(-c)(w,1)$, the condition (c) is equivalent to
    \begin{equation*}
        \psi'(-cy^2) [\omega_{\psi'}^+((w,1))\varphi](y) = [\omega_{\psi'}^+((w,1))\varphi](y),\quad \text{ for all $c\in \mathfrak{p}^{m+\varepsilon}$ and $y\in F$.}
    \end{equation*}
    Suppose that $\operatorname{ord}(y)\leq \left\lceil \frac{c(\psi')-m-\varepsilon}{2} \right\rceil-1$.
    Then $\psi'(-cy^2)\neq1$ for some $c\in\mathfrak{p}^{m+\varepsilon}$.
    Thus $[\omega_{\psi'}^+((w,1))\varphi](y)$ vanishes.
    By \eqref{eqa} and \eqref{eqb}, $[\omega_{\psi'}^+((w,1))\varphi](y)$ is equal, up to a nonzero scalar multiple, to
    \begin{align*}
        \int_F \varphi(x)\psi'(2yx) dx
         & =\int_{\mathfrak{p}^{\left\lceil \frac{\nu}{2} \right\rceil}} \varphi(x)\psi'(2yx) dx                                                              \\
         & =\sum_{i=\left\lceil \frac{\nu}{2} \right\rceil}^\infty q^{-i} \int_{\mathcal{O}^\times} \varphi(\varpi^i u) \psi'(2y\varpi^i u) du                \\
         & =\sum_{i=\left\lceil \frac{\nu}{2} \right\rceil}^\infty q^{-i} \varphi(\varpi^i) \int_{\mathcal{O}^\times} \chi\eta^{-1}(u) \psi'(2y\varpi^i u) du \\
         & =\sum_{i=\left\lceil \frac{\nu}{2} \right\rceil}^\infty q^{-i} \varphi(\varpi^i) g(\chi\eta^{-1}, \psi'_{2y\varpi^i}).
    \end{align*}
    Therefore, we have
    \begin{equation}\label{eqc}
        \sum_{i=\left\lceil \frac{\nu}{2} \right\rceil}^\infty q^{-i} \varphi(\varpi^i) g(\chi\eta^{-1}, \psi'_{2y\varpi^i})=0.
    \end{equation}

    When $\eta=\chi|_{\mathcal{O}^\times}$, Lemma \ref{lemgausssum} tells us that
    \begin{align*}
        g(\chi\eta^{-1}, \psi'_{2y\varpi^i})
        =\begin{cases*}
             1-q^{-1}, & if $i\geq c(\psi')-\operatorname{ord}(y)$,   \\
             -q^{-1},  & if $i= c(\psi')-\operatorname{ord}(y)-1$,    \\
             0,        & if $i\leq c(\psi')-\operatorname{ord}(y)-2$.
         \end{cases*}
    \end{align*}
    Combining this with \eqref{eqc}, $\varphi(\varpi^i)$ is constant for
    \begin{equation*}
        i\geq c(\psi')-\left( \left\lceil \frac{c(\psi')-m-\varepsilon}{2} \right\rceil-1 \right) -1
        =\left\lfloor \frac{m+\nu}{2} \right\rfloor.
    \end{equation*}
    Hence $\varphi$ is determined by
    \begin{equation*}
        \Set{\varphi(\varpi^i) | \left\lceil \frac{\nu}{2} \right\rceil \leq i \leq \left\lfloor \frac{m+\nu}{2} \right\rfloor}.
    \end{equation*}
    Conversely, for any finite sequence $(a_i \mid \lceil\frac{\nu}{2} \rceil \leq i \leq \lfloor \frac{m+\nu}{2}\rfloor)$ of complex numbers, one can construct $\varphi\in\mathcal{S}^+(F)$ satisfying the conditions (a), (b), (c), and $\varphi(\varpi^i)=a_i$.
    Therefore,
    \begin{align*}
        \dim_\C \left( \omega_{\psi,\chi}^+ \right)^{K_m}_\eta
         & =\left\lfloor \frac{m+\nu}{2} \right\rfloor - \left\lceil \frac{\nu}{2} \right\rceil+1 \\
         & =\begin{dcases*}
                1+\left\lfloor \frac{m}{2} \right\rfloor, & if $\nu$ is even, \\
                \left\lfloor \frac{m+1}{2} \right\rfloor, & if $\nu$ is odd.
            \end{dcases*}
    \end{align*}

    When $\eta\neq\chi|_{\mathcal{O}^\times}$, we have $c(\chi\eta^{-1})\geq1$.
    Thus, by Lemma \ref{lemgausssum}, $g(\chi\eta^{-1}, \psi'_{2y\varpi^i})$ is nonzero if and only if $i=c(\psi')-c(\chi\eta^{-1})-\operatorname{ord}(y)$.
    Combining this with \eqref{eqc}, one can see that $\varphi(\varpi^i)=0$ for
    \begin{equation*}
        i\geq c(\psi')-c(\chi\eta^{-1})- \left( \left\lceil \frac{c(\psi')-m-\varepsilon}{2} \right\rceil -1 \right)
        =1-c(\chi\eta^{-1}) + \left\lfloor \frac{m+\nu}{2} \right\rfloor.
    \end{equation*}
    Hence $\varphi$ is determined by
    \begin{equation*}
        \Set{\varphi(\varpi^i) | \left\lceil \frac{\nu}{2} \right\rceil \leq i < 1-c(\chi\eta^{-1}) + \left\lfloor \frac{m+\nu}{2} \right\rfloor}.
    \end{equation*}
    Conversely, for any finite sequence $(a_i \mid \lceil\frac{\nu}{2} \rceil \leq i < 1-c(\chi\eta^{-1}) + \lfloor \frac{m+\nu}{2} \rfloor)$ of complex numbers, one can construct $\varphi\in\mathcal{S}^+(F)$ satisfying the conditions (a), (b), (c), and $\varphi(\varpi^i)=a_i$.
    Therefore,
    \begin{align*}
        \dim_\C \left( \omega_{\psi,\chi}^+ \right)^{K_m}_\eta
         & =\max\left(0, 1-c(\chi\eta^{-1}) + \left\lfloor \frac{m+\nu}{2} \right\rfloor - \left\lceil \frac{\nu}{2} \right\rceil \right) \\
         & =\begin{dcases*}
                \max\left(0, 1+\left\lfloor \frac{m}{2} \right\rfloor -c(\chi\eta^{-1}) \right),  & if $\nu$ is even, \\
                \max\left( 0, \left\lfloor \frac{m+1}{2} \right\rfloor -c(\chi\eta^{-1}) \right), & if $\nu$ is odd.
            \end{dcases*}
    \end{align*}
    This completes the proof.
\end{proof}
\begin{corollary}\label{corevenweil}
    Let $\chi$ be a quadratic or trivial character of $F^\times$ and $\eta$ a character of $\mathcal{O}^\times$ with $\eta(-1)=\chi(-1)$.
    Then we have
    \begin{equation*}
        \dim_\C \left( \omega_{\psi,\chi}^+ \right)^{K_m, \new}_\eta = \begin{cases*}
            1, & if $m=c^\varepsilon_\eta(\omega_{\psi,\chi}^+)=2c(\eta\chi)+c(\chi)$, \\
            0, & otherwise.
        \end{cases*}.
    \end{equation*}
\end{corollary}
\begin{proof}
    This follows immediately from Theorem \ref{thm:rs-Mp2} and Proposition \ref{prop:numgeneric}.
\end{proof}
For any nontrivial additive character $\Psi$ of $F$, the $\Psi$-Whittaker functional on the Schr\"{o}dinger model $(\omega_\Psi^+, \mathcal{S}^+(F))$ is given by the evaluation at 1:
\begin{equation*}
    \varphi \mapsto \varphi(1).
\end{equation*}
\begin{corollary}\label{corwhittakerevenweil}
    Let $\chi$ be a quadratic or trivial character of $F^\times$ and $\eta$ a character of $\mathcal{O}^\times$ such that $\eta(-1)=\chi(-1)$.
    Put $M=c^\varepsilon_\eta(\omega_{\psi,\chi}^+)$.
    Choose the additive character $\psi'\in \{\psi, \psi_\xi, \psi_\varpi, \psi_{\xi\varpi}\}$ such that $\omega_{\psi,\chi}^+$ is isomorphic to $\omega_{\psi'}^+$.
    Then the $\psi'$-Whittaker functional is nonzero on $(\omega_{\psi,\chi}^+)^{K_M}_\eta$.
\end{corollary}
\begin{proof}
    Let us realize the even Weil representation $\omega_{\psi,\chi}^+$ as $(\omega_{\psi'}^+, \mathcal{S}^+(F))$.
    In the proof of Theorem \ref{thmevenweil}, we showed that there exists a function $\varphi \in \mathcal{S}^+(F)$ such that $\varphi(1)=\varphi(\varpi^0)=1$ and $\varphi \in (\omega_{\psi,\chi}^+)^{K_M}_\eta$.
    The assertion now follows.
\end{proof}

We now come to the Steinberg representations.
Recall from Proposition \ref{propreducibility} the exact sequence
\begin{equation*}
    0 \longrightarrow \mathit{St}_{\psi,\chi} \longrightarrow \pi_\psi(\chi \cdot |-|^{\frac{1}{2}}) \overset{\mathcal{M}}{\longrightarrow} \omega_{\psi,\chi}^+ \longrightarrow 0.
\end{equation*}
Here, we put $\mathcal{M}$ to be the surjective map from $\pi_\psi(\chi \cdot |-|^{\frac{1}{2}})$ to $\omega_{\psi,\chi}^+$.
This gives an exact sequence
\begin{equation}\label{eqseqm}
    0 \longrightarrow (\mathit{St}_{\psi,\chi})^{K_m}_\eta \longrightarrow \pi_\psi(\chi \cdot |-|^{\frac{1}{2}})^{K_m}_\eta \overset{\mathcal{M}_m}{\longrightarrow} (\omega_{\psi,\chi}^+)^{K_m}_\eta,
\end{equation}
where we write $\mathcal{M}_m$ for the restriction of $\mathcal{M}$ to $\pi_\psi(\chi \cdot |-|^{\frac{1}{2}})^{K_m}_\eta$.
\begin{lemma}\label{lemintwsurj1}
    Let $\chi$ be a quadratic or trivial character of $F^\times$.
    If $\eta=\chi|_{\mathcal{O}^\times}$, then the map $\mathcal{M}_{c^\varepsilon_\eta(\omega_{\psi,\chi}^+)}$ is surjective.
\end{lemma}
\begin{proof}
    The map $\mathcal{M}$ can be realized as
    \begin{equation*}
        [\mathcal{M} f](g) = \int_F f((w,1)n(x)g) dx.
    \end{equation*}
    Note that the image is not the Schr\"{o}dinger model.
    The maps $\mathcal{M}_m$ ($m>0$) are obtained by the restrictions of $\mathcal{M}$ to $\pi_\psi(|-|^{\frac{1}{2}}\cdot \chi)^{K_m}_\eta$.

    First, we suppose that $\chi$ is unramified.
    Then $\eta=\chi|_{\mathcal{O}^\times}$ is trivial, and Theorem \ref{thmevenweil} implies that $c^\varepsilon_\eta(\omega_{\psi,\chi}^+)=0$.
    Both of $\pi_\psi(|-|^{\frac{1}{2}}\cdot \chi)^{K_0}_1$ and $(\omega_{\psi,\chi}^+)^{K_0}_1$ are 1-dimensional.
    On the other hand, by \cite[Proposition 5.3]{hi}, one can see that $(\mathit{St}_{\psi,\chi})^{K_0}_1=\{0\}$.
    Thus the map $\mathcal{M}_0$ is injective, and hence surjective.

    Next, we consider the case that $\chi$ is ramified.
    Then by Theorem \ref{thmevenweil}, we know that $c^\varepsilon_\eta(\omega_{\psi,\chi}^+)=1$ and $(\omega_{\psi,\chi}^+)^{K_1}_\eta$ is 1-dimensional.
    We shall show that $\mathcal{M}_1\neq0$.
    Lemma \ref{lemcoset} and Corollary \ref{corcoset} imply that
    \begin{equation*}
        \widetilde{G}=\widetilde{B}K_1 \sqcup \widetilde{B}(w,1)K_1.
    \end{equation*}
    Let $f$ be a nonzero element in $\pi_\psi(|-|^{\frac{1}{2}}\cdot \chi)^{K_1}_\eta$ such that
    \begin{equation*}
        f((w,1))\neq0,\ \ f|_{\widetilde{B}K_1}=0.
    \end{equation*}
    Then, decomposing the integral on $F$ into those on $\mathfrak{p}^{-\varepsilon}$ and $\varpi^{-i}\mathcal{O}^\times$ ($i\geq\varepsilon+1$), we have
    \begin{equation*}
        [\mathcal{M}_1 f](1)
        =q^\varepsilon f((w,1)) + \sum_{i=\varepsilon+1}^\infty \int_{\varpi^{-i}\mathcal{O}^\times} f((\left( \begin{array}{cc}&1\\ -1&-x \end{array} \right),1)) dx.
    \end{equation*}
    In the second term, for every $i\geq\varepsilon+1$ and any $x\in \varpi^{-i}\mathcal{O}^\times$, we have
    \begin{equation*}
        f((\left( \begin{array}{cc}&1\\ -1&-x \end{array} \right),1))
        =f((\left( \begin{array}{cc}x^{-1}&-1\\ &x \end{array} \right),1)(\left( \begin{array}{cc}-1&\\-x^{-1}&-1 \end{array} \right),1))
        =0.
    \end{equation*}
    Thus we have
    \begin{equation*}
        [\mathcal{M}_1 f](1) = q^\varepsilon f((w,1)) \neq 0,
    \end{equation*}
    and hence $\mathcal{M}_1 \neq 0$.
    This completes the proof.
\end{proof}
\begin{lemma}\label{lemintwsurj2}
    Let $\chi$ be a quadratic or trivial character of $F^\times$ and $\eta$ a character of $\mathcal{O}^\times$ such that $\eta(-1)=\chi(-1)$.
    If $\eta\neq \chi|_{\mathcal{O}^\times}$, then the map $\mathcal{M}_{c^\varepsilon_\eta(\omega_{\psi,\chi}^+)}$ is surjective.
\end{lemma}
\begin{proof}
    We shall first consider the case $c(\chi)=0$ and $c(\eta)\geq 1$.
    In this case, Theorems \ref{thmps} and \ref{thmevenweil} tell us that
    \begin{equation*}
        c^\varepsilon_\eta(\omega_{\psi,\chi}^+) = c^\varepsilon_\eta(\pi_\psi(\chi \cdot |-|^{\frac{1}{2}})) = 2c(\eta),
    \end{equation*}
    and
    \begin{equation*}
        \dim_\C (\omega_{\psi,\chi}^+)^{K_{2c(\eta)}}_\eta = 1.
    \end{equation*}
    Choose the additive character $\psi'\in \{\psi, \psi_\xi, \psi_\varpi, \psi_{\xi\varpi}\}$ such that $\omega_{\psi,\chi}^+$ is isomorphic to $\omega_{\psi'}^+$.
    By Proposition \ref{corwhittakerps}, there exists a nonzero vector $f\in \pi_\psi(\chi \cdot |-|^{\frac{1}{2}})^{K_{2c(\eta)}}_\eta$ such that $\lambda_{\psi'}(f)\neq 0$.
    Since $\mathit{St}_{\psi,\chi}$ is not $\psi'$-generic, it does not contain $f$.
    This, combining with the exact sequence \eqref{eqseqm}, implies that $\mathcal{M}_{2c(\eta)}(f)\neq 0$.
    Hence $\mathcal{M}_{2c(\eta)}$ is surjective.

    Next we shall consider the case $c(\chi)=1$ and $c(\eta\chi) \geq 1$.
    In this case, Theorems \ref{thmps} and \ref{thmevenweil} tell us that
    \begin{align*}
        c^\varepsilon_\eta(\pi_\psi(\chi \cdot |-|^{\frac{1}{2}})) & = 2c(\eta\chi),     \\
        c^\varepsilon_\eta(\omega_{\psi,\chi}^+)                   & = 2c(\eta\chi) + 1,
    \end{align*}
    and
    \begin{align*}
        \dim_\C \pi_\psi(\chi \cdot |-|^{\frac{1}{2}})^{K_{2c(\eta\chi)}}_\eta   & = 2, \\
        \dim_\C \pi_\psi(\chi \cdot |-|^{\frac{1}{2}})^{K_{2c(\eta\chi)+1}}_\eta & = 4, \\
        \dim_\C \left( \omega_{\psi,\chi}^+ \right)^{K_{2c(\eta\chi)+1}}_\eta    & = 1.
    \end{align*}
    Then the exact sequence in Proposition \ref{propreducibility} (2) implies that $c^\varepsilon_\eta(\mathit{St}_{\psi,\chi})=2c(\eta\chi)$.
    It follows from Theorem \ref{thm:rs-Mp2} and Proposition \ref{prop:numgeneric} that the dimension of $(\mathit{St}_{\psi,\chi})^{K_{2c(\eta\chi)+1}}_\eta$ is less than or equal to $3$.
    Therefore, combining this with \eqref{eqseqm}, we see that $\mathcal{M}_{2c(\eta\chi)+1}$ is nonzero, and hence surjective.
\end{proof}
\begin{lemma}\label{lemintwsurjall}
    Let $\chi$ be a quadratic or trivial character of $F^\times$, and $\eta$ a character of $\mathcal{O}^\times$ such that $\eta(-1)=\chi(-1)$.
    Then the map $\mathcal{M}_m$ is surjective for every $m\geq0$.
\end{lemma}
\begin{proof}
    Put $M=c^\varepsilon_\eta(\omega_{\psi,\chi}^+)$.
    Let $f_0$ be a nonzero element in $\pi_\psi(|-|^{\frac{1}{2}}\cdot \chi)^{K_M}_\eta$ such that $(\omega_{\psi,\chi}^+)^{K_M}_\eta = \C \mathcal{M}_M f_0$, given by Lemmas \ref{lemintwsurj1} or \ref{lemintwsurj2}.
    Then, it follows from Proposition \ref{prop:numgeneric} and Theorem \ref{thm:rs-Mp2} that any element $\varphi\in (\omega_{\psi,\chi}^+)^{K_m}_\eta$ can be written as a linear sum
    \begin{equation*}
        \varphi=\sum_{i=0}^{\left\lfloor \frac{m-M}{2} \right\rfloor} a_i \alpha_2^i(\mathcal{M}_M f_0), \quad a_i \in \C,
    \end{equation*}
    which is equal to
    \begin{equation*}
        \mathcal{M} \left( \sum_{i=0}^{\left\lfloor \frac{m-M}{2} \right\rfloor} a_i \alpha_2^i(f_0) \right).
    \end{equation*}
    Since $f_0 \in \pi_\psi(|-|^{\frac{1}{2}}\cdot \chi)^{K_M}_\eta$, this lies in the image of $\mathcal{M}_m$.
    This completes the proof.
\end{proof}
The lemma gives a short exact sequence
\begin{equation}\label{eqexact}
    0 \longrightarrow (\mathit{St}_{\psi,\chi})^{K_m}_\eta \longrightarrow \pi_\psi(|-|^{\frac{1}{2}}\cdot \chi)^{K_m}_\eta \overset{\mathcal{M}_m}{\longrightarrow} (\omega_{\psi,\chi}^+)^{K_m}_\eta \longrightarrow 0,
\end{equation}
for every $m\geq 0$.
\begin{theorem}\label{thmsteinberg}
    Let $\chi$ be a quadratic or trivial character of $F^\times$, and $\eta$ a character of $\mathcal{O}^\times$ such that $\eta(-1)=\chi(-1)$.
    Then we have
    \begin{align*}
        c^\varepsilon_\eta(\mathit{St}_{\psi,\chi})=\begin{cases*}
                                                        1,            & if $\eta=\chi|_{\mathcal{O}^\times}$,    \\
                                                        2c(\eta\chi), & if $\eta\neq\chi|_{\mathcal{O}^\times}$,
                                                    \end{cases*}
    \end{align*}
    and
    \begin{align*}
        \dim_\C (\mathit{St}_{\psi,\chi})^{K_m}_\eta
         & =\begin{dcases*}
                \left( \left\lceil \frac{3(m-2c(\eta))}{2} \right\rceil +1 -2\delta_{c(\eta)} \right)^+,            & if $\chi$ is unramified, \\
                \left( \left\lceil \frac{3(m-2c(\eta\chi))-1}{2} \right\rceil  +2 -2\delta_{c(\eta\chi)} \right)^+, & if $\chi$ is ramified.
            \end{dcases*}
    \end{align*}
    In particular, $c^\varepsilon_{\min}(\mathit{St}_{\psi,\chi})=c^\varepsilon_\eta(\mathit{St}_{\psi,\chi})$ if and only if $\eta=\chi|_{\mathcal{O}^\times}$, and $c^\varepsilon_{\min}(\mathit{St}_{\psi,\chi})=1$.
\end{theorem}
\begin{proof}
    The short exact sequence \eqref{eqexact} implies that
    \begin{equation*}
        \dim_\C (\mathit{St}_{\psi,\chi})^{K_m}_\eta = \dim_\C \pi_\psi(|-|^{\frac{1}{2}}\cdot \chi)^{K_m}_\eta - \dim_\C (\omega_{\psi,\chi}^+)^{K_m}_\eta.
    \end{equation*}
    The assertions now follow from Theorems \ref{thmps} and \ref{thmevenweil}.
\end{proof}

\begin{corollary}\label{corsteinberg}
    Let $\chi$ be a quadratic or trivial character of $F^\times$, and $\eta$ a character of $\mathcal{O}^\times$ such that $\eta(-1)=\chi(-1)$.
    Put $M=c^\varepsilon_\eta(\mathit{St}_{\psi,\chi})$.
    \begin{enumerate}[(1)]
        \item If $c(\chi)=c(\eta)=0$, then we have $c^\varepsilon_\eta(\mathit{St}_{\psi,\chi})=1$ and
              \begin{align*}
                  (\mathit{St}_{\psi,\chi})^{K_1, \new}_\eta & =1, &  &                     & \\
                  (\mathit{St}_{\psi,\chi})^{K_2, \new}_\eta & =1, &  &                     & \\
                  (\mathit{St}_{\psi,\chi})^{K_3, \new}_\eta & =1, &  &                     & \\
                  (\mathit{St}_{\psi,\chi})^{K_m, \new}_\eta & =0, &  & \text{if } m\geq 4. &
              \end{align*}
        \item If one of $c(\chi)$ or $c(\eta\chi)$ is 0 and the other is not, then we have
              \begin{align*}
                  (\mathit{St}_{\psi,\chi})^{K_M, \new}_\eta     & =1, &  &                       & \\
                  (\mathit{St}_{\psi,\chi})^{K_{M+1}, \new}_\eta & =2, &  &                       & \\
                  (\mathit{St}_{\psi,\chi})^{K_m, \new}_\eta     & =0, &  & \text{if } m\geq M+2. &
              \end{align*}
        \item If $c(\chi)\neq 0 \neq c(\eta\chi)$, then we have
              \begin{align*}
                  (\mathit{St}_{\psi,\chi})^{K_M, \new}_\eta     & =2, &  &                       & \\
                  (\mathit{St}_{\psi,\chi})^{K_{M+1}, \new}_\eta & =1, &  &                       & \\
                  (\mathit{St}_{\psi,\chi})^{K_m, \new}_\eta     & =0, &  & \text{if } m\geq M+2. &
              \end{align*}
    \end{enumerate}
\end{corollary}
\begin{proof}
    This follows immediately from Theorems \ref{thmsteinberg} and \ref{thm:rs-Mp2} and Proposition \ref{prop:numgeneric}.
\end{proof}

\begin{corollary}\label{corwhittakersteinberg}
    Let $\chi$ be a quadratic or trivial character of $F^\times$ and $\eta$ a character of $\mathcal{O}^\times$ such that $\eta(-1)=\chi(-1)$.
    Put $M=c^\varepsilon_\eta(\mathit{St}_{\psi,\chi})$.
    Let $i=0$ or $1$.
    Then, for a nontrivial additive character $\Psi$ of $F$ such that $\mathit{St}_{\psi,\chi}$ is $\Psi$-generic and $c(\Psi)=-\varepsilon-i$, the $\Psi$-Whittaker functional for $\mathit{St}_{\psi,\chi}$ is nonzero on $(\mathit{St}_{\psi,\chi})^{K_{M+i}}_\eta$.
\end{corollary}
\begin{proof}
    We shall write $\mathcal{L}$ for the canonical surjective mapping from $\pi_\psi(\chi\cdot|-|^{-\frac{1}{2}})$ to $\mathit{St}_{\psi,\chi}$.
    The short exact sequence
    \begin{equation*}
        0 \longrightarrow \omega_{\psi,\chi}^+ \longrightarrow \pi_\psi(\chi\cdot|-|^{-\frac{1}{2}}) \overset{\mathcal{L}}{\longrightarrow} \mathit{St}_{\psi,\chi} \longrightarrow 0,
    \end{equation*}
    in Proposition \ref{propreducibility} (3) gives an exact sequence
    \begin{equation*}
        0 \longrightarrow \left( \omega_{\psi,\chi}^+ \right)^{K_m}_\eta \longrightarrow \pi_\psi(\chi\cdot|-|^{-\frac{1}{2}})^{K_m}_\eta \overset{\mathcal{L}_m}{\longrightarrow} \left( \mathit{St}_{\psi,\chi} \right)^{K_m}_\eta,
    \end{equation*}
    for any $m\geq 0$, where $\mathcal{L}_m$ denotes the restriction of $\mathcal{L}$ to $\pi_\psi(\chi\cdot|-|^{-\frac{1}{2}})^{K_m}_\eta$.
    Let $\Psi$ be a nontrivial additive character of $F$ such that $\mathit{St}_{\psi,\chi}$ is $\Psi$-generic, and $\Lambda_\Psi$ a $\Psi$-Whittaker functional for $\mathit{St}_{\psi,\chi}$.
    Then the composition $\Lambda_\Psi \circ \mathcal{L}$ is a $\Psi$-Whittaker functional for $\pi_\psi(\chi\cdot|-|^{-\frac{1}{2}})$, which is equal to $\lambda_\Psi$ up to a scalar multiple.
    Therefore, if there exists $f\in \pi_\psi(\chi\cdot|-|^{-\frac{1}{2}})^{K_m}_\eta$ such that $\lambda_\Psi(f)\neq 0$, then $v=\mathcal{L}_m(f) \in (\mathit{St}_{\psi,\chi})^{K_m}_\eta$ satisfies $\Lambda_\Psi(v)\neq 0$.
    Now the assertion follows from Corollary \ref{corwhittakerps}.
\end{proof}

\section{Supercuspidal representations}\label{secsc}
In this section we shall consider irreducible genuine supercuspidal representations of $\widetilde{G}=\widetilde{\SL_2}(F)$.
For this we need some propositions on how they are constructed.
In addition, we also need those on how the irreducible supercuspidal representations of $G=\SL_2(F)$ are constructed and how those constructions are related to each other.
\subsection{Classifications of supercuspidal representations}\label{sectype}
We begin by reviewing the constructions and the classifications of irreducible supercuspidal representations of $\widetilde{G}$ and $G$ given by Manderscheid\cite{man1} and those of $G$ given by Kutzko--Sally\cite{ks}.
Let $\delta\in\{0,1\}$.
For $l\geq1$, let $J^\delta_l$ be a subgroup of $K^\delta$ defined by
\begin{align*}
    J^\delta_l = \Set{\left( \begin{array}{cc}a&b\\ c&d \end{array} \right) \in G | a, d\in 1+\mathfrak{p}^l,\ b\in\mathfrak{p}^{l-\delta},\ c\in\mathfrak{p}^{l+\delta} }.
\end{align*}
Put $J^\delta_0=K^\delta$.
For any integer $j$, let $N_j$ be the subgroup of $N$ consisting of $n(b)$ where $b\in\mathfrak{p}^j$.

Let $\sigma$ be a finite dimensional representation of $K^\delta$.
Then $\sigma$ is trivial on $J^\delta_m$ for some $m\geq0$.
We shall write $c(\sigma)$ for the minimum of such $m$, and call it the conductor of $\sigma$.
A representation $\sigma$ is said to be strongly cuspidal if
\begin{align*}
    \Hom_{N_{c(\sigma)-2\delta-1}} (\sigma|_{N_{c(\sigma)-2\delta-1}}, 1)=\{0\},
\end{align*}
where $1$ denotes the trivial representation of $N_{c(\sigma)-2\delta-1}$.
Moreover, a strongly cuspidal representation $\sigma$ of $K^\delta$ is said to have defect 1 if
\begin{align*}
    \Hom_{N_{c(\sigma)-\delta-1}} (\sigma|_{N_{c(\sigma)-\delta-1}}, 1)\neq\{0\},
\end{align*}
where $1$ denotes the trivial representation of $N_{c(\sigma)-\delta-1}$.
Otherwise, we say that $\sigma$ has defect 0.
Note that if the defect of $\sigma$ is 1 then $\delta$ is 1.
For two representations $\sigma_1$ and $\sigma_2$ of $K^{\delta_1}$ and $K^{\delta_2}$ respectively, we say that $\sigma_1$ and $\sigma_2$ are equivalent if $\delta_1=\delta_2$ and $\sigma_1$ and $\sigma_2$ are equivalent as representations of $K^{\delta_1}=K^{\delta_2}$.
Manderscheid \cite{man1} gave the following constructions and classifications of irreducible supercuspidal representations of $\SL_2(F)$ and $\widetilde{\SL_2}(F)$.
\begin{theorem}[\cite{man1}]\label{thmmansl2}
    For any irreducible strongly cuspidal representation $\sigma$ of $K^\delta$, $\cInd^G_{K^\delta}(\sigma)$ is an irreducible supercuspidal representation of $G$, and every irreducible supercuspidal representation of $G$ can be obtained in this way.
    Moreover, inequivalent irreducible strongly cuspidal representations compactly induce inequivalent irreducible supercuspidal representations.
\end{theorem}
The splitting $x\mapsto(x,\bm{s}^\delta(x))$ over $K^\delta$ gives a bijective correspondence between the representations of $K^\delta$ and the genuine representations of $\widetilde{K}^\delta$.
For an irreducible strongly cuspidal representation $\sigma$ of $K^\delta$, let $\sigma\bm{s}^\delta$ denote the genuine irreducible representation $(x,\epsilon\bm{s}^\delta(x)) \mapsto \epsilon\sigma(x)$.
\begin{theorem}[\cite{man1}]
    For any irreducible strongly cuspidal representation $\sigma$ of $K^\delta$, $\cInd^{\widetilde{G}}_{\widetilde{K}^\delta}(\sigma\bm{s}^\delta)$ is an irreducible genuine supercuspidal representation of $\widetilde{G}$, and every irreducible genuine supercuspidal representation of $\widetilde{G}$ can be obtained in this way.
    Moreover, inequivalent irreducible strongly cuspidal representations compactly induce inequivalent irreducible genuine supercuspidal representations.
\end{theorem}

Next, we shall review the result of Kutzko--Sally\cite{ks} after recalling Kutzko's construction(\cite{kut1, kut2}) of minimal supercuspidal representations of $\GL_2(F)$.
See also \cite{car}, \cite{bk} or \cite[Chapter 4]{bh}.
Put
\begin{align*}
    I=I_0 & =\Set{\left( \begin{array}{cc}a&b\\ c&d \end{array} \right) \in \GL_2(F) | a,d \in \mathcal{O}^\times,\ b\in\mathcal{O},\ c\in\mathfrak{p}}, \\
    H=H_0 & =\GL_2(\mathcal{O}).
\end{align*}
For $l\geq1$, let $I_l$ and $H_l$ be subgroups of $I$ and $H$, respectively, defined by
\begin{align*}
    I_l & =\Set{\left( \begin{array}{cc}a&b\\ c&d \end{array} \right) \in I | a,d \in 1+\mathfrak{p}^l,\ b\in\mathfrak{p}^l,\ c\in\mathfrak{p}^{l+1}}, \\
    H_l & =\Set{\left( \begin{array}{cc}a&b\\ c&d \end{array} \right) \in H | a,d \in 1+\mathfrak{p}^l,\ b,c\in\mathfrak{p}^l}.
\end{align*}
Let $Z$ be the center of $\GL_2(F)$ and $Z'$ the subgroup of $\GL_2(F)$ generated by
\begin{align*}
    \left( \begin{array}{cc}&1\\ \varpi& \end{array} \right).
\end{align*}
Consider an irreducible finite dimensional representation $\rho$ of $ZH$ or $Z'I$.
It is said to be very cuspidal of level $l\geq1$ if the following two conditions hold:
\begin{enumerate}
    \item $\rho$ is trivial on $H_l$ (resp. $I_l$) if it is a representation of $ZH$ (resp. $Z'I$);
    \item $\Hom_{N_{l-1}} (\rho|_{N_{l-1}}, 1)=\{0\}$,
\end{enumerate}
where $1$ denotes the trivial representation of $N_{l-1}$.

An irreducible admissible representation $\tau$ of $\GL_2(F)$ is said to be minimal if
\begin{equation*}
    c(\tau) \leq c(\tau \otimes(\chi\circ\det))
\end{equation*}
for any character $\chi$ of $F^\times$, where $c(\tau)$ denotes the conductor of $\tau$ defined by Casselman \cite{cas}.
Kutzko \cite{kut1,kut2} gave the following classifications of irreducible minimal supercuspidal representations of $\GL_2(F)$.
\begin{proposition}\label{propkutclas}
    Let $\Gamma=ZH$ or $Z'I$.
    For any irreducible very cuspidal representation $\rho$ of $\Gamma$, $\cInd^{\GL_2(F)}_\Gamma(\rho)$ is an irreducible minimal supercuspidal representation of $\GL_2(F)$, and every irreducible minimal supercuspidal representation of $\GL_2(F)$ can be obtained in this way.
    Moreover, inequivalent irreducible very cuspidal representations compactly induce inequivalent irreducible minimal supercuspidal representations.
\end{proposition}
An irreducible minimal supercuspidal representation $\cInd^{\GL_2(F)}_\Gamma(\rho)$ is said to be unramified (resp. ramified) if $\Gamma=ZH$ (resp. $Z'I$).
\begin{proposition}\label{propkutcond}
    If $\rho$ is an irreducible very cuspidal representation of $\Gamma=ZH$ (resp. $Z'I$) of level $l$, then the conductor of $\cInd^{\GL_2(F)}_\Gamma(\rho)$ is $2l$ (resp. $2l+1$).
\end{proposition}

Let us now recall the result of Kutzko--Sally\cite{ks}.
We shall begin with the unramified case.
Let $\rho$ be an irreducible very cuspidal representation of $ZH$ of level $l$, and $\sigma$ its restriction to $ZH\cap G= K^0$.
Put $\tau=\cInd^{\GL_2(F)}_{ZH}(\rho)$, which is an irreducible supercuspidal representation of $\GL_2(F)$.
Suppose first that $\sigma$ is irreducible.
Then two representations $\cInd^G_{K^0}(\sigma)$ and $\cInd^G_{K^1}(\prescript{\beta}{}{\sigma})$ are irreducible, and the restriction of $\tau$ to $G$ is a direct sum of those two irreducible supercuspidal representations.
The set of those two is called an unramified supercuspidal $L$-packet of cardinality two.
Suppose next that $\sigma$ is reducible.
In this case, we have $l=1$ and  $\sigma=\sigma_1 \oplus \sigma_2$, where $\sigma_1$ and $\sigma_2$ are irreducible representations of dimension $(q-1)/2$.
Then the restriction of $\tau$ to $G$ is a direct sum of four irreducible supercuspidal representations $\cInd^G_{K^0}(\sigma_1)$, $\cInd^G_{K^1}(\prescript{\beta}{}{\sigma_1})$, $\cInd^G_{K^0}(\sigma_2)$, and $\cInd^G_{K^1}(\prescript{\beta}{}{\sigma_2})$.
The set of those four is called the unramified supercuspidal $L$-packet of cardinality four.

Now we come to the ramified case.
Let $\rho$ be an irreducible very cuspidal representation of $Z'I$ of level $l$, and $\sigma$ its restriction to $Z'I\cap G= I\cap G$.
Put $\tau=\cInd^{\GL_2(F)}_{Z'I}(\rho)$, which is an irreducible supercuspidal representation of $\GL_2(F)$.
Then we have $\sigma=\sigma_1 \oplus \sigma_2$, where $\sigma_1$ and $\sigma_2$ are irreducible representations of $I\cap G$.
The restriction of $\tau$ to $G$ is a direct sum of two irreducible supercuspidal representations $\cInd^G_{I\cap G}(\sigma_1)$ and $\cInd^G_{I\cap G}(\sigma_2)$.
The set of those two is called a ramified supercuspidal $L$-packet.

Every element in those supercuspidal $L$-packets is irreducible supercuspidal representation of $G$.
Conversely, every irreducible supercuspidal representation of $G$ can be obtained in this way.

The construction of Manderscheid is compatible with that of Kutzko--Sally.
The proof of the following lemma is straightforward.
\begin{lemma}\label{lem:cptbmanks}
    \begin{enumerate}[(1)]
        \item Let $\rho$ be an irreducible very cuspidal representation of $ZH$ of level $l$, and $\sigma$ an irreducible component of its restriction to $ZH\cap G= K^0$.
              Then $\sigma$ and $\prescript{\beta}{}{\sigma}$ are strongly cuspidal representations of conductor $l$ and defect $0$, of $K^0$ and $K^1$, respectively.
        \item Let $\rho$ be an irreducible very cuspidal representation of $Z'I$ of level $l$, and $\sigma$ an irreducible component of its restriction to $Z'I\cap G= I\cap G$.
              Then $\cInd^{K^1}_{I\cap G}(\sigma)$ is an irreducible strongly cuspidal representation of $K^1$ of conductor $l+1$ and defect $1$.
    \end{enumerate}
\end{lemma}
\begin{proposition}\label{prop:compatibilityManKS}
    Let $\pi$ be an irreducible supercuspidal representation of $G$.
    Assume that $\rho$ is an irreducible very cuspidal representation of $\Gamma\in\{ZH, Z'I\}$ such that $\pi$ is a subrepresentation of $\cInd^{\GL_2(F)}_\Gamma(\rho)$.
    In addition, assume that $\sigma$ is an irreducible strongly cuspidal representation of $K^\delta$ ($\delta=0,1$), such that $\pi=\cInd^G_{K^\delta}(\sigma)$.
    Then we have
    \begin{align*}
        l(\rho) & =c(\sigma)-d(\sigma),     \\
        \Gamma  & =\begin{cases*}
                       ZH,  & if $d(\sigma)=0$, \\
                       Z'I, & if $d(\sigma)=1$,
                   \end{cases*}
    \end{align*}
    where $l(\rho)$ and $d(\sigma)$ denote the level of $\rho$ and the defect of $\sigma$, respectively.
\end{proposition}
\begin{proof}
    The proposition follows from Lemma\ref{lem:cptbmanks} and Theorem \ref{thmmansl2}.
\end{proof}

\subsection{Some results of Lansky--Raghuram}\label{subsec:resultsofLR}
We shall recall some results of Lansky--Raghuram \cite{lr} on supercuspidal representations of $G=\SL_2(F)$.
For an irreducible strongly cuspidal representation $\sigma$ of $K^\delta$ ($\delta=0,1$), there exists $z(\sigma)\in\{\pm1\}$ such that $\sigma(-1_2)$ is a scalar multiplication by $z(\sigma)$.
We shall call $z(\sigma)$ the central sign of $\sigma$.
Note that the central character of $\cInd^G_{K^\delta}(\sigma)$ is  given by $-1 \mapsto z(\sigma)$.
\begin{theorem}\label{thm:lrsc}
    Let $\varepsilon, \delta\in\{0,1\}$, and put $K_m=K^\varepsilon_m$ and $K=K^\varepsilon$.
    Let $\sigma$ be an irreducible strongly cuspidal representation of $K^\delta$, and put $\pi=\cInd^G_{K^\delta}(\sigma)$, which is an irreducible supercuspidal representation of $G$.
    Let $\eta$ be a character of $\mathcal{O}^\times$ such that $\eta(-1)=z(\sigma)$.
    \begin{enumerate}[(1)]
        \item If the defect of $\sigma$ is 0, then $\pi^{K_{2c(\sigma)-1}}_\eta=\{0\}$. Moreover, if $c(\eta)\leq c(\sigma)$, then the followings hold.
              \begin{enumerate}[(I)]
                  \item If $\pi$ does not belong to the unramified supercuspidal $L$-packet of cardinality four, then we have
                        \begin{align*}
                            \dim_\C \pi^{K_m}_\eta=\begin{dcases*}
                                                       0,                                                    & if $m\leq 2c(\sigma)-1$,                                         \\
                                                       2\left\lfloor \frac{m-2c(\sigma)+1}{2} \right\rfloor, & if $c(\sigma)+\varepsilon+\delta$ is odd and $m\geq2c(\sigma)$,  \\
                                                       2\left\lceil \frac{m-2c(\sigma)+1}{2} \right\rceil,   & if $c(\sigma)+\varepsilon+\delta$ is even and $m\geq2c(\sigma)$.
                                                   \end{dcases*}
                        \end{align*}
                  \item If $\pi$ belongs to the unramified supercuspidal $L$-packet of cardinality four, then $c(\sigma)=1$ and we have
                        \begin{align*}
                            \dim_\C \pi^{K_m}_\eta=\begin{dcases*}
                                                       0,                                        & if $m\leq 1$,                                  \\
                                                       \left\lfloor \frac{m-1}{2} \right\rfloor, & if $\varepsilon+\delta$ is even and $m\geq 2$, \\
                                                       \left\lceil \frac{m-1}{2} \right\rceil,   & if $\varepsilon+\delta$ is odd and $m\geq 2$.
                                                   \end{dcases*}
                        \end{align*}
              \end{enumerate}
        \item If the defect of $\sigma$ is 1, then $\pi^{K_{2c(\sigma)-2}}_\eta=\{0\}$. Moreover, if $c(\eta)\leq c(\sigma)-1$, then we have
              \begin{align*}
                  \dim_\C \pi^{K_m}_\eta=\begin{cases*}
                                             0,              & if $m\leq 2c(\sigma)-2$, \\
                                             m-2c(\sigma)+2, & if $m\geq 2c(\sigma)-1$.
                                         \end{cases*}
              \end{align*}
    \end{enumerate}
\end{theorem}
\begin{proof}
    The theorem is an immediate consequence of \cite[Propositions 3.3.4, 3.3.6, 3.3.8]{lr} and Proposition \ref{prop:compatibilityManKS}.
\end{proof}

\begin{theorem}\label{thm:lrwhittakersc}
    Let $\varepsilon, \delta\in\{0,1\}$, and put $K_m=K^\varepsilon_m$, $K=K^\varepsilon$, and $\psi=\psi^\varepsilon$.
    Let $\sigma$ be an irreducible strongly cuspidal representation of $K^\delta$, and put $\pi=\cInd^G_{K^\delta}(\sigma)$, which is an irreducible supercuspidal representation of $G$.
    Let $\eta$ be a character of $\mathcal{O}^\times$ such that $\eta(-1)=z(\sigma)$.
    Let $M$ be a positive integer such that $\pi^{K_{M-1}}_\eta \{0\}$ and $\pi^{K_M}_\eta \neq \{0\}$.
    \begin{enumerate}[(1)]
        \item If the defect of $\sigma$ is 0 and $c(\eta)\leq c(\sigma)$, then the followings hold.
              \begin{enumerate}[(I)]
                  \item If $\pi$ does not belong to the unramified supercuspidal $L$-packet of cardinality four, then $\pi^{K_M}_\eta$ is $2$-dimensional, and $\pi$ is $\Psi$-generic if and only if $c(\Psi)+\delta$ is even.
                        Moreover, there does not exist a nonzero vector in $\pi^{K_M}_\eta$ at which both of a $\psi_{\xi\varpi^{|\delta-\varepsilon|}}$-Whittaker functional and a $\psi_{\varpi^{|\delta-\varepsilon|}}$-Whittaker functional vanish.
                  \item If $\pi$ belongs to the unramified supercuspidal $L$-packet of cardinality four, then $\pi^{K_M}_\eta$ is $1$-dimensional, and $\pi$ is $\Psi$-generic for exactly one of $\Psi=\psi^\delta$ or $\psi^\delta_\xi$. Let $\Psi\in\{\psi, \psi_\xi, \psi_\varpi, \psi_{\xi\varpi}\}$ be an additive character such that $\pi$ is $\Psi$-generic. Then no representation in the packet other than $\pi$ is $\Psi$-generic, and moreover a $\Psi$-Whittaker functional does not vanish on $\pi^{K_M}_\eta$.
              \end{enumerate}
        \item If the defect of $\sigma$ is 1 and $c(\eta)\leq c(\sigma)-1$, then $\pi^{K_M}_\eta$ is $1$-dimensional, $\pi^{K_{M+1}}_\eta$ is $2$-dimensional, and $\pi$ is $\Psi$-generic for exactly one of $\Psi=\psi$ or $\psi_\xi$. In addition, $\pi$ is also $\Psi'$-generic for exactly one of $\Psi'=\psi_\varpi$ or $\psi_{\xi\varpi}$. Let $\Psi \in \{\psi, \psi_\xi, \psi_\varpi, \psi_{\xi\varpi}\}$ be an additive character such that $\pi$ is $\Psi$-generic. Then the other representation in the ramified supercuspidal packet containing $\pi$ is $\Psi_\xi$-generic. Moreover, a $\Psi$-Whittaker functional does not vanish on $\pi^{K_M}_\eta$ (resp. $\pi^{K_{M+1}}_\eta$) if $c(\Psi)+\varepsilon$ is even (resp. odd).
    \end{enumerate}
\end{theorem}
\begin{proof}
    The dimensions of $\pi^{K_m}_\eta$ are given in Theorem \ref{thm:lrsc}.
    The assertions on the genericities are shown in Proposition 3.3.5, Corollary 3.3.7, and Proposition 3.3.9 in \cite{lr}.
    We need to show the assertions on Whittaker functionals.
    The proofs are similar to those of \cite[Proposition 3.3.5, Corollary 3.3.7, and Proposition 3.3.9]{lr}.
    We shall consider the assertion $(1)-(I)$ in the case of $\varepsilon=0$.
    The others can be proven analogously.

    Let $\widetilde{\pi}$ be the minimal irreducible supercuspidal representation of $\GL_2(F)$ whose restriction to $G$ contains $\pi$, and $(K(\widetilde{\pi}, \Psi), C^\infty_0(F^\times))$ its Kirillov model with respect to $\Psi$, for any additive character $\Psi$.
    Let $\omega$ be the central character of $\widetilde{\pi}$.
    Since $\eta(-1)=z(\sigma)$, the kernel of $\eta^{-1}\omega|_{\mathcal{O}^\times}$ contains $\{\pm1\}$.
    Then there exists a character $\chi$ of $\mathcal{O}^{\times2}$ such that $\eta^{-1}(a)\omega(a)=\chi(a^2)$ for any $a\in\mathcal{O}^\times$.
    We shall define $\phi_\eta \in C^\infty_0(F^\times)$ by
    \begin{align*}
        \phi_\eta(x)= \begin{cases*}
                          \chi(x), & if $x \in \mathcal{O}^{\times2}$,    \\
                          0,       & if $x \notin \mathcal{O}^{\times2}$.
                      \end{cases*}
    \end{align*}
    Combined with \cite[Proposition 3.4]{tun}, the argument of \cite[Proposition 3.3.5]{lr} shows that $K(\widetilde{\pi}, \psi^0)(K^0_{2c(\sigma)})$ and $K(\widetilde{\pi}, \psi^0_\varpi)(K^0_{2c(\sigma)+1})$ act on $\phi_\eta$ via $\eta$, and the assertion $(1)-(I)$ holds when $\varepsilon=0$.
\end{proof}

\subsection{Main results for supercuspidal representations}
Now we shall calculate the dimensions of $\pi^{K^\varepsilon_m}_\eta$ for supercuspidal representations $\pi$.
Let $\varepsilon\in\{0,1\}$.
Put $\psi=\psi^\varepsilon$ and $K_m=K^\varepsilon_m$.
Note that $\bm{s}^\delta\bm{s}^\varepsilon$ can be regarded as a character on $K^\delta$.
The central sign of $\cInd^{\widetilde{G}}_{\widetilde{K}^\delta}(\sigma\bm{s}^\delta)$ with respect to $\psi$ is equal to that of $\sigma \otimes(\bm{s}^\delta\bm{s}^\varepsilon)$, and that of $\cInd^G_{K^\delta}(\sigma)$ is that of $\sigma$.
\begin{lemma}\label{lemgeneric}
    Let $\sigma$ be an irreducible strongly cuspidal representation of $K^\delta$ ($\delta=0,1$).
    Put $\pi=\cInd^{\widetilde{G}}_{\widetilde{K}^\delta}(\sigma\bm{s}^\delta)$ and $\pi'=\cInd^G_{K^\delta}(\sigma)$, which are irreducible supercuspidal representations of $\widetilde{G}$ and $G$, respectively.
    Then, for any nontrivial additive character $\Psi$ of $F$, $\pi$ is $\Psi$-generic if and only if $\pi'$ is $\Psi$-generic.
\end{lemma}
\begin{proof}
    Using Frobenius reciprocity and Mackey theory (\cite[Corollary 2.7 (1)]{yam}), we have
    \begin{align*}
        \Hom_N(\pi, \Psi)
         & =\Hom_{\widetilde{G}}(\cInd^{\widetilde{G}}_{\widetilde{K}^\delta} (\sigma\bm{s}^\delta), \Ind^{\widetilde{G}}_N (\Psi))                                               \\
         & \cong \prod_{g\in \widetilde{K}^\delta \backslash \widetilde{G} / N} \Hom_{\widetilde{K}^\delta \cap \prescript{g}{}{N}} (\sigma\bm{s}^\delta, \prescript{g}{}{\Psi}).
    \end{align*}
    Similarly, we have
    \begin{equation*}
        \Hom_N(\pi', \Psi)
        \cong \prod_{g\in K^\delta \backslash G / N} \Hom_{K^\delta \cap \prescript{g}{}{N}} (\sigma, \prescript{g}{}{\Psi}).
    \end{equation*}

    We shall identify $\widetilde{K}^\delta \backslash \widetilde{G} / N$ with $K^\delta \backslash G / N$.
    For any $g \in K^\delta \backslash G / N$, the intersection $K^\delta \cap \prescript{g}{}{N}$ is an open compact subgroup of $\prescript{g}{}{N}\cong F$, and thus isomorphic to $\mathcal{O}$.
    Since $\mathcal{O}$ has no nontrivial quadratic character, the splittings $K^\delta\hookrightarrow \widetilde{G}$ and $\prescript{g}{}{N}\hookrightarrow \widetilde{G}$ coincide on $K^\delta \cap \prescript{g}{}{N}$.
    More precisely, $\bm{s}^\delta$ is trivial on $K^\delta \cap \prescript{g}{}{N}$, and we have
    \begin{equation*}
        \widetilde{K}^\delta \cap \prescript{g}{}{N}
        = \Set{ (x,1) | x\in K^\delta \cap \prescript{g}{}{N}}.
    \end{equation*}
    Here, note that $\prescript{g}{}{N}=\{(\prescript{g}{}{n},1) \mid n\in N\}$ as subgroups of $\widetilde{G}$, since the conjugation action of $G$ (or $\widetilde{G}$) on $\widetilde{G}$ preserves the second component.
    Therefore, we have
    \begin{equation*}
        \Hom_{\widetilde{K}^\delta \cap \prescript{g}{}{N}} (\sigma\bm{s}^\delta, \prescript{g}{}{\Psi})
        = \Hom_{K^\delta \cap \prescript{g}{}{N}} (\sigma, \prescript{g}{}{\Psi}).
    \end{equation*}
    Thus, we have
    \begin{align*}
        \Hom_N(\pi, \Psi) \cong \Hom_N(\pi', \Psi).
    \end{align*}
    This completes the proof.
\end{proof}

Every irreducible strongly cuspidal representation of $K^\delta$ with conductor 1 can be identified with an irreducible representation of $\SL_2(\mathcal{O}/\mathfrak{p})$ whose restriction to
\begin{equation*}
    \Set{\left( \begin{array}{cc}1&b\\ &1 \end{array} \right) \in \SL_2(\mathcal{O}/\mathfrak{p}) | b \in \mathcal{O}/\mathfrak{p}}
\end{equation*}
is nontrivial.
Indeed, we have isomorphisms
\begin{equation*}
    K^1/K^1_1 \overset{\cong}{\longrightarrow} K^0/K^0_1 \overset{\cong}{\longrightarrow} \SL_2(\mathcal{O}/\mathfrak{p}),
\end{equation*}
where the first isomorphism is given by $x\mapsto \beta^{-1} x \beta$, and the second one is a natural reduction map.
Tanaka \cite{tan} gave a classification of irreducible representations of $\SL_2(\mathcal{O}/\mathfrak{p})$, which tells us that there are exactly two such representations with dimension $\frac{1}{2}(q-1)$, up to equivalent.
\begin{proposition}
    Let $\sigma^\delta_1$ and $\sigma^\delta_2$ be the two inequivalent irreducible strongly cuspidal $(q-1)/2$-dimensional representations of $K^\delta$ with conductor 1, for each $\delta=0,1$.
    Then $\cInd^{\widetilde{G}}_{\widetilde{K}^\delta} (\sigma^\delta_i\bm{s}^\delta)$, ($\delta=0,1$, $i=1,2$) are all four odd Weil representations.
    Moreover, if $\cInd^{\widetilde{G}}_{\widetilde{K}^\delta} (\sigma^\delta_i\bm{s}^\delta)=\omega_\psi^-$ for some additive character $\psi$ of $F$, then $\delta+1 \equiv c(\psi)$ modulo 2.
\end{proposition}
\begin{proof}
    Fix a nontrivial additive character $\Psi$ of $F$.
    By \cite[Corollary 3.3.7]{lr}, an irreducible supercuspidal representation $\cInd^G_{K^\delta} (\sigma^\delta_i)$ ($\delta=0,1$, $i=1,2$) is $\Psi_a$-generic for only one $a\in F^\times/F^{\times2}$, and different $\cInd^G_{K^\delta} (\sigma^\delta_i)$ are $\Psi_a$-generic for different $a \in F^\times/F^{\times2}$.
    Now the assertion follows from Lemma \ref{lemgeneric} and the characterization of the odd Weil representations (\S\ref{secrep}).
\end{proof}

\begin{theorem}\label{thmsc}
    Let $\varepsilon\in\{0,1\}$, and put $K_m=K^\varepsilon_m$.
    Let $\delta\in\{0,1\}$.
    Consider an irreducible genuine supercuspidal representation $\pi=\cInd^{\widetilde{G}}_{\widetilde{K}^\delta}(\sigma\bm{s}^\delta)$ of $\widetilde{G}$, where $\sigma$ is an irreducible strongly cuspidal representation $K^\delta$.
    Let $\eta$ be a character of $\mathcal{O}^\times$ such that the central character of $\sigma$ is equal to $\eta|_{\{\pm1\}}$.
    \begin{enumerate}[(1)]
        \item If the defect of $\sigma$ is 0, then $\pi^{K_m}_\eta=\{0\}$ for $m\leq 2c(\sigma)-1$. If moreover $c(\eta)\leq c(\sigma)$, then the followings hold.
              \begin{enumerate}[(I)]
                  \item If $\pi$ is not isomorphic to any odd Weil representation, then we have
                        \begin{align*}
                            \dim_\C \pi^{K_m}_\eta=\begin{dcases*}
                                                       2\left\lfloor \frac{m-2c(\sigma)+1}{2} \right\rfloor, & if $c(\sigma)+\varepsilon+\delta$ is odd and $m\geq2c(\sigma)$,  \\
                                                       2\left\lceil \frac{m-2c(\sigma)+1}{2} \right\rceil,   & if $c(\sigma)+\varepsilon+\delta$ is even and $m\geq2c(\sigma)$.
                                                   \end{dcases*}
                        \end{align*}
                        In particular, $c^\varepsilon_{\min}(\pi)=c^\varepsilon_\eta(\pi)$ if $c(\eta)\leq c(\sigma)$, and $c^\varepsilon_{\min}(\pi)=2c(\sigma)$ (resp. $2c(\sigma)+1$) if $c(\sigma)+\varepsilon+\delta$ is even (resp. odd).
                  \item If $\pi$ is isomorphic to $\omega_{\Psi}^-$ for some $\Psi$, then $c(\sigma)=1$ and we have
                        \begin{align*}
                            \dim_\C \pi^{K_m}_\eta=\begin{dcases*}
                                                       \left\lfloor \frac{m-1}{2} \right\rfloor, & if $\varepsilon+\delta$ is even and $m\geq 2$, \\
                                                       \left\lceil \frac{m-1}{2} \right\rceil,   & if $\varepsilon+\delta$ is odd and $m\geq 2$.
                                                   \end{dcases*}
                        \end{align*}
                        In particular, $c^\varepsilon_{\min}(\pi)=c^\varepsilon_\eta(\pi)$ if $c(\eta)\leq 1$, and $c^\varepsilon_{\min}(\pi)=2$ (resp. $3$) if $\varepsilon+\delta$ is odd (resp. even).
              \end{enumerate}
        \item If the defect of $\sigma$ is 1, then $\pi^{K_m}_\eta=\{0\}$ for $m\leq 2c(\sigma)-2$. If moreover $c(\eta)\leq c(\sigma)-1$, then we have
              \begin{equation*}
                  \dim_\C \pi^{K_m}_\eta = m-2c(\sigma)+2,
              \end{equation*}
              for $m\geq 2c(\sigma)-1$.
              In particular, $c^\varepsilon_{\min}(\pi)=c^\varepsilon_\eta(\pi)$ if $c(\eta)\leq c(\sigma)-1$, and $c^\varepsilon_{\min}(\pi)=2c(\sigma)-1$.
    \end{enumerate}
\end{theorem}
\begin{proof}
    Using Frobenius reciprocity and Mackey theory (\cite[Corollary 2.7 (2)]{yam}), we have
    \begin{align*}
        \pi^{K_m}_\eta
         & \cong\Hom_{K_m}(\eta, \cInd^{\widetilde{G}}_{\widetilde{K}^\delta}(\sigma\bm{s}^\delta))                                                                                     \\
         & =\Hom_{\widetilde{G}}(\cInd^{\widetilde{G}}_{K_m}(\eta), \cInd^{\widetilde{G}}_{\widetilde{K}^\delta}(\sigma\bm{s}^\delta))                                                  \\
         & \cong \bigoplus_{g\in K_m\backslash \widetilde{G}/\widetilde{K}^\delta} \Hom_{K_m\cap \prescript{g}{}{(\widetilde{K}^\delta)}} (\eta, \prescript{g}{}{\sigma\bm{s}^\delta}).
    \end{align*}
    Similarly, we put $\pi'=\cInd^{G}_{K^\delta}(\sigma)$ and obtain
    \begin{align*}
        (\pi')^{K_m}_\eta
        \cong \bigoplus_{g\in K_m\backslash G/K^\delta} \Hom_{K_m\cap \prescript{g}{}{(K^\delta)}} (\eta, \prescript{g}{}{\sigma}).
    \end{align*}
    Since the splittings $K_m \hookrightarrow \widetilde{G}$ and $\prescript{g}{}{K^\delta}\hookrightarrow \widetilde{G}$ coincide on $K_m\cap \prescript{g}{}{(K^\delta)}$ for any $g\in K_m\backslash G/K^\delta$, we have
    \begin{equation*}
        \bigoplus_{g\in K_m\backslash \widetilde{G}/\widetilde{K}^\delta} \Hom_{K_m\cap \prescript{g}{}{(\widetilde{K}^\delta)}} (\eta, \prescript{g}{}{\sigma\bm{s}^\delta})
        \cong \bigoplus_{g\in K_m\backslash G/K^\delta} \Hom_{K_m\cap \prescript{g}{}{(K^\delta)}} (\eta, \prescript{g}{}{\sigma}).
    \end{equation*}
    Therefore, we have
    \begin{equation*}
        \dim_\C \pi^{K_m}_\eta = \dim_\C (\pi')^{K_m}_\eta.
    \end{equation*}
    Now the assertions follow from \cite[3.3]{lr} and Proposition \ref{prop:compatibilityManKS}.
\end{proof}
\begin{corollary}\label{corsc}
    Let $\eta$ be a character of $\mathcal{O}^\times$ such that the central character of $\sigma$ is equal to $\eta|_{\{\pm1\}}$.
    Assume that $c(\eta)\leq c(\sigma)-d(\sigma)$, where $d(\sigma)$ denotes the defect of $\sigma$.
    Put $M=c^\varepsilon_\eta(\pi)$.
    \begin{enumerate}[(1)]
        \item If the defect of $\sigma$ is 0, then $\pi^{K_m, \new}_\eta$ is zero unless $m=M$, in which case we have
              \begin{align*}
                  \dim_\C \pi^{K_M, \new}_\eta
                  =\begin{cases*}
                       2, & if $\pi$ is not isomorphic to any odd Weil representation,   \\
                       1, & if $\pi$ is isomorphic to $\omega_{\Psi}^-$ for some $\Psi$.
                   \end{cases*}
              \end{align*}
        \item If the defect of $\sigma$ is 1, then $\pi^{K_m, \new}_\eta$ is zero unless $m=M, M+1$, in which cases we have
              \begin{equation*}
                  \dim_\C \pi^{K_M, \new}_\eta =\dim_\C \pi^{K_{M+1}, \new}_\eta = 1.
              \end{equation*}
    \end{enumerate}
\end{corollary}
\begin{proof}
    This follows immediately from Theorems \ref{thmsc} and \ref{thm:rs-Mp2} and Proposition \ref{prop:numgeneric}.
\end{proof}

We shall consider the multiplicity one property for newforms in terms of the Whittaker functionals.
The main idea for the proof of the following lemma was provided by Yamamoto.
\begin{lemma}\label{lem:Yamamotowhittakersupercuspidal}
    Let $\varepsilon\in\{0,1\}$.
    Let $\sigma$ be an irreducible strongly cuspidal representation of $K^\delta$ ($\delta=0,1$), and $\Psi$ a nontrivial additive character of $F$.
    Put $\pi=\cInd^{\widetilde{G}}_{\widetilde{K}^\delta}(\sigma\bm{s}^\delta)$ and $\pi'=\cInd^G_{K^\delta}(\sigma)$, which are irreducible supercuspidal representations of $\widetilde{G}$ and $G$, respectively.
    Let $\eta$ be a character of $\mathcal{O}^\times$ such that $\eta(-1)=z(\sigma)$.
    Put $\eta'=\eta \cdot (\chi_{\psi^\varepsilon}\chi_{\psi^\delta})$.
    Assume that both of $\pi$ and $\pi'$ are $\Psi$-generic.
    Then, for any $m\geq c(\eta)$, the $\Psi$-Whittaker functional for $\pi$ is nonzero on $\pi^{K^\varepsilon_m}_\eta$ if and only if that for $\pi'$ is nonzero on $(\pi')^{K^\varepsilon_m}_{\eta'}$.
\end{lemma}
\begin{proof}
    The $\Psi$-Whittaker functional for $\pi$ is nonzero on $\pi^{K^\varepsilon_m}_\eta$ if and only if $\lambda \circ f$ is nonzero for some $\lambda \in \Hom_N(\pi, \Psi)$ and $f\in \Hom_{K^\varepsilon_m}(\eta, \pi)$.
    Similarly, the $\Psi$-Whittaker functional for $\pi'$ is nonzero on $(\pi')^{K^\varepsilon_m}_{\eta'}$ if and only if $\lambda' \circ f'$ is nonzero for some $\lambda' \in \Hom_N(\pi', \Psi)$ and $f'\in \Hom_{K^\varepsilon_m}(\eta', \pi')$.
    We shall consider the composition maps $\lambda \circ f$ and $\lambda' \circ f'$.

    We shall fix a complete set of representatives $\mathcal{R}(\widetilde{K}^\delta \backslash \widetilde{G} / N)$ for $\widetilde{K}^\delta \backslash \widetilde{G} / N$.
    It follows from the argument in the proof of \cite[Corollary 2.7 (1)]{yam} that a linear isomorphism
    \begin{equation*}
        \Hom_N(\pi, \Psi) \cong \prod_{g\in \mathcal{R}(\widetilde{K}^\delta \backslash \widetilde{G} / N)} \Hom_{\widetilde{K}^\delta \cap \prescript{g}{}{N}} (\sigma\bm{s}^\delta, \prescript{g}{}{\Psi}), \quad \lambda \mapsto \mathcal{F}(\lambda) = \left( \mathcal{F}_g(\lambda) \right)_{g\in \mathcal{R}(\widetilde{K}^\delta \backslash \widetilde{G} / N)},
    \end{equation*}
    which we used in the proof of Lemma \ref{lemgeneric}, is given by
    \begin{equation*}
        [\mathcal{F}_g(\lambda)](w) = \lambda(\pi(g^{-1}) \cdot \phi_w),
    \end{equation*}
    for $\lambda \in \Hom_N(\pi, \Psi)$ and $w\in \sigma$, where $\phi_w$ is an element in $\pi$ given by $\phi_w(x)=\bm{s}^\delta(x) \sigma(x)w$ if $x\in \widetilde{K}^\delta$, and $\phi_w(x)=0$ if $x\in \widetilde{G}\setminus \widetilde{K}^\delta$.

    On the other hand, we shall fix a complete set of representatives $\mathcal{R}(K^\varepsilon_m \backslash \widetilde{G} / \widetilde{K}^\delta)$ for $K^\varepsilon_m \backslash \widetilde{G} / \widetilde{K}^\delta$.
    It follows from the argument in the proof of \cite[Corollary 2.7 (2)]{yam} that a linear isomorphism
    \begin{equation*}
        \Hom_{K^\varepsilon_m}(\eta, \pi) \cong \bigoplus_{h \in \mathcal{R}(K^\varepsilon_m \backslash \widetilde{G} / \widetilde{K}^\delta)} \Hom_{K^\varepsilon_m \cap \prescript{h}{}{(\widetilde{K}^\delta)}} (\eta, \prescript{h}{}{(\sigma \bm{s}^\delta)}), \quad f \mapsto \mathcal{G}(f)=\left(\mathcal{G}_h(f)\right)_{h\in \mathcal{R}(K^\varepsilon_m \backslash \widetilde{G} / \widetilde{K}^\delta)},
    \end{equation*}
    which we used in the proof of Theorem \ref{thmsc}, is given by
    \begin{equation*}
        [\mathcal{G}_h(f)](u)= [f(u)](h^{-1}),
    \end{equation*}
    for $u\in \C$.
    For any $u\in \C$, $y\in K^\varepsilon_m$, $h\in \mathcal{R}(K^\varepsilon_m \backslash \widetilde{G} / \widetilde{K}^\delta)$, $x\in \widetilde{K}^\delta$, $g\in \mathcal{R}(\widetilde{K}^\delta \backslash \widetilde{G} / N)$, and $n \in N$, we have
    \begin{align*}
        \left[ \Psi(n) \circ \mathcal{F}_g(\lambda) \circ (\sigma\bm{s}^\delta)(x) \circ \mathcal{G}_h(f) \circ \eta(y) \right](u)
         & = \lambda \left( \pi(n g^{-1}) \cdot \phi_{[(\sigma\bm{s}^\delta)(x) \circ \mathcal{G}_h(f) \circ \eta(y)](u)} \right), \\
         & = \lambda \left( \pi(n g^{-1} x) \cdot \phi_{f(\eta(y)u)(h^{-1})} \right).
    \end{align*}
    Moreover, a calculation shows that
    \begin{equation*}
        \pi(n g^{-1} x) \phi_{f(\eta(y)u)(h^{-1})} = \mathbf{1}_{\widetilde{K}^\delta x^{-1} g n^{-1}} \cdot \pi(n g^{-1} x h^{-1} y) f(u),
    \end{equation*}
    where $\mathbf{1}_{\widetilde{K}^\delta x^{-1} g n^{-1}}$ denotes the characteristic function of $\widetilde{K}^\delta x^{-1} g n^{-1}$.

    For each $h\in \mathcal{R}(K^\varepsilon_m \backslash \widetilde{G} / \widetilde{K}^\delta)$, we shall fix a complete set of representatives $\mathcal{R}(K^\varepsilon_m \cap \prescript{h}{}{\widetilde{K}^\delta} \backslash K^\varepsilon_m)$.
    Moreover, for each $h\in \mathcal{R}(K^\varepsilon_m \backslash \widetilde{G} / \widetilde{K}^\delta)$ and $y \in \mathcal{R}(K^\varepsilon_m \cap \prescript{h}{}{\widetilde{K}^\delta} \backslash K^\varepsilon_m)$, we shall write $x_{(h,y)}$, $g_{(h,y)}$, and $n_{(h,y)}$ for the unique elements in $\widetilde{K}^\delta$, $\mathcal{R}(\widetilde{K}^\delta \backslash \widetilde{G} / N)$, and $N$, respectively, such that $x^{-1} g n^{-1} = h^{-1} y$.
    We shall now define a linear map
    \begin{equation*}
        \Xi
        \colon \left( \prod_{g\in \mathcal{R}(\widetilde{K}^\delta \backslash \widetilde{G} / N)} \Hom_{\widetilde{K}^\delta \cap \prescript{g}{}{N}} (\sigma\bm{s}^\delta, \prescript{g}{}{\Psi}) \right)
        \otimes_\C \left( \bigoplus_{h \in \mathcal{R}(K^\varepsilon_m \backslash \widetilde{G} / \widetilde{K}^\delta)} \Hom_{K^\varepsilon_m \cap \prescript{h}{}{(\widetilde{K}^\delta)}} (\eta, \prescript{h}{}{(\sigma \bm{s}^\delta)}) \right)
        \to \Hom_\C(\C, \C),
    \end{equation*}
    by
    \begin{equation*}
        \Xi(\mathscr{F} \otimes \mathscr{G})
        = \sum_{h\in \mathcal{R}(K^\varepsilon_m \backslash \widetilde{G} / \widetilde{K}^\delta)}
        \sum_{y \in \mathcal{R}(K^\varepsilon_m \cap \prescript{h}{}{\widetilde{K}^\delta} \backslash K^\varepsilon_m)}
        \Psi(n_{(h,y)}) \circ \mathscr{F}_{g_{(h,y)}} \circ (\sigma\bm{s}^\delta)(x_{(h,y)}) \circ \mathscr{G}_h \circ \eta(y),
    \end{equation*}
    where $\mathscr{F}_g$ denotes the $g$-component of $\mathscr{F}$, and $\mathscr{G}_h$ the $h$-component of $\mathscr{G}$.
    Then we have
    \begin{align*}
        \Xi(\mathcal{F}(\lambda) \otimes \mathcal{G}(f))
         & =\sum_{h\in \mathcal{R}(K^\varepsilon_m \backslash \widetilde{G} / \widetilde{K}^\delta)}
        \sum_{y \in \mathcal{R}(K^\varepsilon_m \cap \prescript{h}{}{\widetilde{K}^\delta} \backslash K^\varepsilon_m)}
        \lambda \left( \mathbf{1}_{\widetilde{K}^\delta h^{-1} y} \cdot f(u) \right)                 \\
         & = \lambda \circ f,
    \end{align*}
    for any $\lambda \in \Hom_N(\pi, \Psi)$ and $f\in \Hom_{K^\varepsilon_m}(\eta, \pi)$.
    Hence, $\lambda \circ f$ is nonzero for some $\lambda$ and $f$ if and only if $\Xi$ is nonzero.

    Similarly, we have two linear isomorphisms
    \begin{align*}
        \mathcal{F}'
         & \colon
        \Hom_N(\pi', \Psi) \overset{\cong}{\longrightarrow} \prod_{g\in \mathcal{R}(K^\delta \backslash G/ N)} \Hom_{K^\delta \cap \prescript{g}{}{N}} (\sigma, \prescript{g}{}{\Psi}), \\
        \mathcal{G}'
         & \colon
        \Hom_{K^\varepsilon_m}(\eta', \pi') \overset{\cong}{\longrightarrow} \bigoplus_{h \in \mathcal{R}(K^\varepsilon_m \backslash G / K^\delta)} \Hom_{K^\varepsilon_m \cap \prescript{h}{}{(K^\delta)}} (\eta', \prescript{h}{}{(\sigma)}),
    \end{align*}
    given by \cite[Corollary 2.7]{yam}, and a linear map
    \begin{equation*}
        \Xi'
        \colon \left( \prod_{g\in \mathcal{R}(K^\delta \backslash G / N)} \Hom_{K^\delta \cap \prescript{g}{}{N}} (\sigma, \prescript{g}{}{\Psi}) \right)
        \otimes_\C \left( \bigoplus_{h \in \mathcal{R}(K^\varepsilon_m \backslash G / K^\delta)} \Hom_{K^\varepsilon_m \cap \prescript{h}{}{(K^\delta)}} (\eta', \prescript{h}{}{\sigma}) \right)
        \to \Hom_\C(\C, \C),
    \end{equation*}
    defined by
    \begin{equation*}
        \Xi'(\mathscr{F}' \otimes \mathscr{G}')
        = \sum_{h\in \mathcal{R}(K^\varepsilon_m \backslash G / K^\delta)}
        \sum_{y \in \mathcal{R}(K^\varepsilon_m \cap \prescript{h}{}{K^\delta} \backslash K^\varepsilon_m)}
        \Psi(n_{(h,y)}) \circ \mathscr{F}'_{g_{(h,y)}} \circ \sigma(x_{(h,y)}) \circ \mathscr{G}'_h \circ \eta'(y),
    \end{equation*}
    for $\mathscr{F}'=(\mathscr{F}'_g)_{g\in \mathcal{R}(K^\delta \backslash G / N)}$ and $\mathscr{G}'=(\mathscr{G}'_h)_{h \in \mathcal{R}(K^\varepsilon_m \backslash G / K^\delta)}$.
    Here $\mathcal{R}(K^\delta \backslash G / N)$ (resp. $\mathcal{R}(K^\varepsilon_m \backslash G / K^\delta)$, resp. $\mathcal{R}(K^\varepsilon_m \cap \prescript{h}{}{K^\delta} \backslash K^\varepsilon_m)$) is the image of $\mathcal{R}(\widetilde{K}^\delta \backslash \widetilde{G} / N)$ (resp. $\mathcal{R}(K^\varepsilon_m \backslash \widetilde{G} / \widetilde{K}^\delta)$, resp. $\mathcal{R}(K^\varepsilon_m \cap \prescript{h}{}{\widetilde{K}^\delta} \backslash K^\varepsilon_m)$) under the covering map $\widetilde{G}\to G$, and, by slight abuse of notation, we again write $x_{(h,y)}$, $g_{(h,y)}$, and $n_{(h,y)}$ for the images of them.
    Then we have
    \begin{equation*}
        \Xi'(\mathcal{F}'(\lambda') \otimes \mathcal{G}'(f'))
        = \lambda' \circ f',
    \end{equation*}
    for any $\lambda' \in \Hom_N(\pi', \Psi)$ and $f'\in \Hom_{K^\varepsilon_m}(\eta', \pi')$.
    Hence, $\lambda' \circ f'$ is nonzero for some $\lambda'$ and $f'$ if and only if $\Xi'$ is nonzero.

    As in the proof of Lemma \ref{lemgeneric} and Theorem \ref{thmsc}, we have
    \begin{align*}
        \Hom_{\widetilde{K}^\delta \cap \prescript{g}{}{N}} (\sigma\bm{s}^\delta, \prescript{g}{}{\Psi})
         & = \Hom_{K^\delta \cap \prescript{g}{}{N}} (\sigma, \prescript{g}{}{\Psi}),     \\
        \Hom_{K_m\cap \prescript{h}{}{(\widetilde{K}^\delta)}} (\eta, \prescript{h}{}{(\sigma\bm{s}^\delta)})
         & = \Hom_{K_m\cap \prescript{h}{}{(K^\delta)}} (\eta', \prescript{h}{}{\sigma}).
    \end{align*}
    Thus $\Xi'$ equals $\Xi$.
    This completes the proof.
\end{proof}

\begin{corollary}\label{cor:whittakersupercuspidal}
    Let $\varepsilon, \delta\in\{0,1\}$, and put $K_m=K^\varepsilon_m$, $K=K^\varepsilon$, and $\psi=\psi^\varepsilon$.
    Let $\sigma$ be an irreducible strongly cuspidal representation of $K^\delta$, and put $\pi=\cInd^{\widetilde{G}}_{\widetilde{K}^\delta}(\sigma\bm{s}^\delta)$, which is an irreducible supercuspidal representation of $\widetilde{G}$.
    Let $\eta$ be a character of $\mathcal{O}^\times$ such that $\eta(-1)=z(\sigma)$.
    Put $M=c^\varepsilon_\eta(\pi)$.
    \begin{enumerate}[(1)]
        \item If the defect of $\sigma$ is 0 and $c(\eta)\leq c(\sigma)$, then the followings hold.
              \begin{enumerate}[(I)]
                  \item If $\pi$ is not isomorphic to any odd Weil representation, then $\pi^{K_M}_\eta$ is $2$-dimensional, and $\pi$ is $\Psi$-generic if and only if $c(\Psi)+\delta$ is even.
                        Moreover, there does not exist a nonzero vector in $\pi^{K_M}_\eta$ at which both of a $\psi_{\varpi^{|\delta-\varepsilon|}}$-Whittaker functional and a $\psi_{\xi\varpi^{|\delta-\varepsilon|}}$-Whittaker functional vanish.
                  \item If $\pi$ is isomorphic to $\omega_\Psi^-$, for some $\Psi\in\{\psi, \psi_\xi, \psi_\varpi, \psi_{\xi\varpi}\}$, then $\pi^{K_M}_\eta$ is $1$-dimensional. Moreover, a $\Psi$-Whittaker functional does not vanish on $\pi^{K_M}_\eta$.
              \end{enumerate}
        \item If the defect of $\sigma$ is 1 and $c(\eta)\leq c(\sigma)-1$, then $\pi^{K_M}_\eta$ is $1$-dimensional, $\pi^{K_{M+1}}_\eta$ is $2$-dimensional, and $\pi$ is $\Psi$-generic for exactly one of $\Psi=\psi$ or $\psi_\xi$. In addition, $\pi$ is also $\Psi'$-generic for exactly one of $\Psi'=\psi_\varpi$ or $\psi_{\xi\varpi}$. Let $\Psi \in \{\psi, \psi_\xi, \psi_\varpi, \psi_{\xi\varpi}\}$ be an additive character such that $\pi$ is $\Psi$-generic. Then a $\Psi$-Whittaker functional does not vanish on $\pi^{K_M}_\eta$ (resp. $\pi^{K_{M+1}}_\eta$) if $c(\Psi)+\varepsilon$ is even (resp. odd).
    \end{enumerate}
\end{corollary}
\begin{proof}
    The corollary is an immediate consequence of Theorem \ref{thm:lrwhittakersc} and Lemma \ref{lem:Yamamotowhittakersupercuspidal}.
\end{proof}

\section{Behaviour under the local theta correspondence}\label{sectheta}
In the last two sections, the conductors of irreducible genuine representations of $\widetilde{G}$ are given explicitly.
In this section, we shall describe how the conductors for $\widetilde{G}$ behave under the local theta correspondence.

Let $\varepsilon\in\{0,1\}$, and put $\psi=\psi^\varepsilon$.
We fix them throughout this section.
We shall identify representations of $\PGL_2(F)$ (resp. $\PSL_2(F)$) with those of $\GL_2(F)$ (resp. $\SL_2(F)$) with trivial central character.
In particular, we define the conductor of an irreducible infinite-dimensional representation of $\PGL_2(F)$ to be that of its pullback to $\GL_2(F)$.
Let $D$ be the quaternion division algebra over $F$.
Then Waldspurger \cite{wal1,wal2} showed that the local theta correspondence gives a bijection
\begin{equation*}
    \theta_\psi \colon \Irr(\widetilde{G}) \longrightarrow \Irr(\PGL_2(F)) \sqcup \Irr(\mathit{PD}^\times),
\end{equation*}
which depends on the choice of the additive character $\psi$.
Here, $\Irr(A)$ denotes the set of equivalence classes of irreducible (genuine) smooth admissible representations of $A=\widetilde{G}$, $\PGL_2(F)$, or $\mathit{PD}^\times$.
For an irreducible genuine representation $\pi$ of $\widetilde{G}$, the conductor of its theta lift $\theta_\psi(\pi)$ is defined by Casselman \cite{cas} if it is an irreducible representation of $\PGL_2(F)$.

First we shall consider non-supercuspidal representations.
For two characters $\mu_1$ and $\mu_2$ of $F^\times$, we shall write $\pi(\mu_1,\mu_2)$ for a representation of $\GL_2(F)$ on the space of locally constant functions $f \colon \GL_2(F) \to \C$ such that
\begin{equation*}
    f(\left( \begin{array}{cc}a&b\\ &d \end{array} \right)g)=\left| \frac{a}{d} \right|^{\frac{1}{2}} \mu_1(a)\mu_2(d) f(g),
\end{equation*}
for any $a,d\in F^\times$, $b\in F$, and $g\in\GL_2(F)$, with the action $[g\cdot f](x)=f(xg)$.
The central character of $\pi(\mu_1,\mu_2)$ is trivial if and only if $\mu_1=\mu_2^{-1}$.
It is known that $\pi(\mu_1,\mu_2)$ is irreducible if and only if $\mu_1\mu_2^{-1}\neq|-|^{\pm1}$.
For a character $\mu$ of $F^\times$, write $\mathit{St}_\mu$ for the unique irreducible subrepresentation of $\pi(\mu|-|^{\frac{1}{2}}, \mu|-|^{-\frac{1}{2}})$.
It is called a (twisted) Steinberg representation of $\GL_2(F)$.
The central character of $\mathit{St}_\mu$ is trivial if and only if $\mu^2=1$.

\begin{proposition}\label{propthetanonsc}
    \begin{enumerate}[(1)]
        \item Let $\mu$ be a character of $F^\times$ such that $\mu^2\neq |-|^{\pm1}$.
              Then $\pi_\psi(\mu)$ is an irreducible $\psi$-generic representation of $\widetilde{G}$, and its theta lift $\theta_\psi(\pi_\psi(\mu))$ is isomorphic to an irreducible representation $\pi(\mu,\mu^{-1})$ of $\PGL_2(F)$.
              In particular, the conductor of $\theta_\psi(\pi_\psi(\mu))$ is $2c(\mu)$.
        \item Let $\chi$ be a quadratic or trivial character of $F^\times$.
              Then the even Weil representation $\omega_{\psi,\chi}^+$.
              Its theta lift $\theta_\psi(\omega_{\psi,\chi}^+)$ is isomorphic to a 1-dimensional irreducible representation $\chi \circ \det$ of $\PGL_2(F)$.
        \item For a quadratic or trivial character $\chi$ of $F^\times$, the following three conditions are equivalent:
              \begin{enumerate}[(i)]
                  \item $\chi$ is not trivial;
                  \item $\mathit{St}_{\psi,\chi}$ is $\psi$-generic;
                  \item $\theta_\psi(\mathit{St}_{\psi,\chi})$ is a representation of $\PGL_2(F)$.
              \end{enumerate}
              Let $\chi$ be a nontrivial quadratic character of $F^\times$.
              Then $\theta_\psi(\mathit{St}_{\psi,\chi})$ is the twisted Steinberg representation $\mathit{St}_\chi$, whose conductor is $1$ (resp. $2$) if $\chi$ is unramified (resp. ramified).
    \end{enumerate}
\end{proposition}
\begin{proof}
    The proposition is a result of Waldspurger \cite{wal1, wal2}.
\end{proof}

Next we consider supercuspidal representations.
\begin{proposition}\label{propthetaoddweil}
    For $a\in F^\times$ and an odd Weil representation $\omega_{\psi_a}^-$, the following three conditions are equivalent:
    \begin{enumerate}[(i)]
        \item $a\in F^{\times2}$;
        \item $\omega_{\psi_a}^-$ is $\psi$-generic;
        \item $\theta_\psi(\omega_{\psi_a}^-)$ is a representation of $\PGL_2(F)$.
    \end{enumerate}
    Moreover, $\theta_\psi(\omega_\psi^-)$ is the Steinberg representation $\mathit{St}$, whose conductor is 1.
\end{proposition}

\begin{lemma}\label{lemmin}
    Every irreducible supercuspidal representation of $\GL_2(F)$ with trivial central character is minimal.
\end{lemma}
\begin{proof}
    Let $\tau$ be a minimal irreducible supercuspidal representation of $\GL_2(F)$, and $\chi$ a character of $F^\times$.
    Then $\tau\otimes(\chi\circ\det)$ is also irreducible and supercuspidal, and every irreducible supercuspidal representation of $\GL_2(F)$ is of this form.
    Put $\tau'=\tau\otimes(\chi\circ\det)$, and suppose that $\tau'$ is not minimal and has trivial central character.
    It follows from \cite[Proposition 3.4]{tun} that
    \begin{align*}
        c(\tau')=2c(\chi)>c(\tau),
    \end{align*}
    where $c(\tau)$ and $c(\tau')$ denote the conductors (\cite{cas}) of $\tau$ and $\tau'$, respectively.
    Let $\rho$ be an irreducible very cuspidal representation of $\Gamma=ZH$ or $Z'I$ of level $l\geq1$, such that $\tau=\cInd^{\GL_2(F)}_\Gamma(\rho)$.
    Then by Proposition \ref{propkutcond}, we have $c(\tau)=2l$ or $2l+1$.
    Hence we have $c(\chi)\geq l+1$.
    For any $u\in1+\mathfrak{p}^l$, we have
    \begin{align*}
        \tau'\left( \left( \begin{array}{cc}u&\\ &u \end{array} \right) \right)=1,
    \end{align*}
    since the central character of $\tau'$ is trivial.
    On the other hand, this equals
    \begin{align*}
         & \tau\left( \left( \begin{array}{cc}u&\\ &u \end{array} \right) \right) \cdot \chi\left( \det\left( \begin{array}{cc}u&\\ &u \end{array} \right) \right) \\
         & =\rho\left( \left( \begin{array}{cc}u&\\ &u \end{array} \right) \right) \cdot \chi(u^2)                                                                 \\
         & =\chi(u^2),
    \end{align*}
    since $\rho$ is of level $l$.
    Recall that the residual characteristic $F$ is assumed to be odd.
    Thus a mapping $u\mapsto u^2$ gives an automorphism of $1+\mathfrak{p}^l$.
    Therefore, $\chi$ is trivial on $1+\mathfrak{p}^l$, i.e., $c(\chi)\leq l$.
    This is a contradiction.
\end{proof}

\begin{proposition}\label{propthetasc}
    Let $\delta\in\{0,1\}$, and $\sigma$ be an irreducible strongly cuspidal representation of $K^\delta$.
    Consider an irreducible genuine supercuspidal representation $\pi=\cInd^{\widetilde{G}}_{\widetilde{K}^\delta}(\sigma\bm{s}^\delta)$ of $\widetilde{G}$.
    Assume that $\pi$ is not isomorphic to any odd Weil representation.
    Then, $\theta_\psi(\pi)$ is a representation of $\PGL_2(F)$ if and only if $\pi$ is $\psi$-generic.
    Suppose in addition that $\pi$ is $\psi$-generic.
    Then, the conductor of $\theta_\psi(\pi)$ is equal to $c^\varepsilon_{\min}(\pi)$.
\end{proposition}
\begin{proof}
    The first assertion is a result of Waldspurger \cite{wal1,wal2}.
    Suppose that $\pi$ is $\psi$-generic.
    Then $\cInd^G_{K^\delta}(\sigma)$ is also $\psi$-generic (Lemma \ref{lemgeneric}).
    By the Propositions in \cite[\S3.3]{lr}, we have
    \begin{equation}\label{eqsc}
        c^\varepsilon_{\min}(\pi)=c(\cInd^G_{K^\delta}(\sigma)),
    \end{equation}
    where the right hand side is the conductor for $G=\SL_2(F)$ defined by Lansky--Raghuram \cite{lr}.
    It follows from Manderscheid's result \cite[Theorems 2.7, 2.8, 3.5, 3.6, 3.7]{man3} that there exist $\delta'\in\{0,1\}$ and an irreducible strongly cuspidal representation $\sigma'$ of $K^{\delta'}$ such that
    \begin{itemize}
        \item[(a)] the central character of $\sigma'$ is trivial,
        \item[(b)] the conductor and the defect of $\sigma'$ are equal to those of $\sigma$, respectively,
        \item[(c)] the restriction of $\theta_\psi(\pi)$ to $\PSL_2(F)$ contains $\sigma'$, and
        \item[(d)] $\theta_\psi(\pi)$ is an irreducible component of $\Ind^{\PGL_2(F)}_{\PSL_2(F)}(\pi')$,
    \end{itemize}
    where $\pi'$ denotes the irreducible supercuspidal representation $\cInd^G_{K^{\delta'}}(\sigma')$.
    Here, note that $\pi'$ can be regarded as a representation of $\PSL_2(F)$ since the central character of $\pi'$ is trivial by (a).
    We regard $\theta_\psi(\pi)$ as a representation of $\GL_2(F)$.
    It is irreducible and supercuspidal.
    By Lemma \ref{lemmin}, we know that $\theta_\psi(\pi)$ is minimal.
    Moreover, the conditions (c) and (d) imply that the restriction of $\theta_\psi(\pi)$ to $G$ contains $\pi'$.
    Then by \cite[Proposition 3.4.1]{lr} and the condition (b), the conductor of $\theta_\psi(\pi)$ equals $c(\pi')$ and the right hand side of \eqref{eqsc}.
    This completes the proof.
\end{proof}

As a consequence, we have:
\begin{theorem}\label{thmtheta}
    Let $\pi$ be an irreducible $\psi$-generic genuine representation of $\widetilde{G}$
    In particular, $\theta_\psi(\pi)$ is a representation of $\PGL_2(F)$.
    Assume that $z_\psi(\pi)=+1$.
    Then the conductor of $\theta_\psi(\pi)$ is equal to $c^\varepsilon_1(\pi)$.
\end{theorem}
\begin{proof}
    The conductor of $\theta_\psi(\pi)$ is given in Propositions \ref{propthetanonsc}, \ref{propthetaoddweil} and \ref{propthetasc}.
    On the other hand, $c^\varepsilon_1(\pi)$ is given in Theorems \ref{thmps}, \ref{thmevenweil}, \ref{thmsteinberg} and \ref{thmsc}.
\end{proof}
\begin{remark}
    The odd Weil representation $\omega_\psi^-$ is an irreducible genuine $\psi$-generic representation of $\widetilde{G}$, but has the central sign $-1$ with respect to $\psi$.
    For any $\eta$, we have $c^\varepsilon_\eta(\omega_\psi^-)=2$ and $c^{1-\varepsilon}_\eta(\omega_\psi^-)=3$ if they exist.
    On the other hand, the conductor of $\theta_\psi(\omega_\psi^-)$ is 1.
\end{remark}
\begin{remark}
    Let $\mathit{JL}$ denotes the Jacquet--Langlands correspondence
    \begin{equation*}
        \mathit{JL} \colon \Irr(\mathit{PD}^\times) \hookrightarrow \Irr(\PGL_2(F)).
    \end{equation*}
    For any $\pi\in\Irr(\widetilde{G})$ such that $\theta_\psi(\pi)\in\Irr(\mathit{PD}^\times)$, put
    \begin{equation*}
        \mathit{Wd}_\psi(\pi) = \mathit{JL} \circ \theta_\psi (\pi) \in \Irr(\PGL_2(F)).
    \end{equation*}
    Although the local theta correspondence $\pi\mapsto \theta_\psi(\pi)$ preserves the conductors in the sense of Theorem \ref{thmtheta}, the correspondence $\mathit{Wd}_\psi$ does not preserve them.
    The counterexamples appear in the supercuspidal representations which are induced by strongly cuspidal representations of defect 0.
\end{remark}


\begin{thebibliography}{99}
    \bibitem{al}A. O. L. Atkin and J. Lehner, {\it Hecke operators on $\Gamma_0(m)$}, Math. Ann. {\bf 185} (1970), pp. 134-160.
    \bibitem{aky}H. Atobe, S. Kondo and S. Yasuda, {\it Local newforms for the general linear groups over a non-archimedean local field}, Forum Math. Pi {\bf 10} (2022), e24.
    \bibitem{bh}C. J. Bushnell and G. Henniart, {\it The Local Langlands Conjecture for $\GL(2)$}, Grundlehren der mathematischen Wissenschaften {\bf 335} Springer, (2006).
    \bibitem{bk}C. J. Bushnell and P. C. Kutzko, The admissible dual of GL(N) via compact open subgroups, Annals of Math. Studies, vol. 129, Princeton University Press, (1993).
    \bibitem{car}H. Carayol, {\it Repr\'{e}sentations cuspidales du groupe lin\'{e}aire}, Ann. Scient. \'{E}c. Norm. Sup. (4) {\bf 17} (1984), pp. 191-225.
    \bibitem{cas}W. Casselman, {\it On some results of Atkin and Lehner}, Math. Ann. {\bf 201} (1973), pp. 301-314.
    \bibitem{cheng}Y. Cheng, {\it Rankin--Selberg integrals for $\SO_{2n+1}\times\GL_r$ attached to newforms and oldforms}, Math. Z. {\bf 301} (2022), pp. 3973-4014.
    \bibitem{hi}K. Hiraga and T. Ikeda, {\it On the Kohnen plus space for Hilbert modular forms of half-integral weight I}, Compos. Math. {\bf 149} (2013), pp. 1963-2010.
    \bibitem{jpss}H. Jacquet, I. I. Piatetski-Shapiro and J. Shalika, {\it Conducteur des repr\'{e}sentations du groupe lin\'{e}aire}, Math. Ann. {\bf 256} (1981), no. 2, pp. 199-214.
    \bibitem{koh}W. Kohnen, {\it Modular forms of half integral weight on $\Gamma_0(4)$}, Math. Ann., {\bf248} (1980), pp. 249-266.
    \bibitem{kohnew}W. Kohnen, {\it Newforms of half-integral weight}, J. Reine Angew. Math. {\bf 333} (1982), pp. 32-72.
    \bibitem{kut1}P.C. Kutzko, {\it On the supercuspidal representations of $Gl_2$}, Amer. J. Math. {\bf 100} (1978), pp. 43-60.
    \bibitem{kut2}P. C. Kutzko, {\it On the supercuspidal representations of $Gl_2$, II}, Amer. J. Math. {\bf 100} (1978), 705-716.
    \bibitem{ks}P. C. Kutzko and P. J. Sally,, Jr., {\it All supercuspidal representations of $\SL_\ell$ over a $p$-adic field are induced}, In: Trombi, P.C. (eds) {\it Representation Theory of Reductive Groups}, Progress in Mathematics, vol. 40, Birkh\"{a}user Boston, (1983).
    \bibitem{lr}J. Lansky and A. Raghuram, {\it Conductors and newforms for $\SL(2)$}, Pacific J. Math. {\bf 231} (2007) (1), pp. 127-153.
    \bibitem{man1}D. Manderscheid, {\it On the supercuspidal representations of $\SL_2$ and its two-fold cover. I}, Math. Ann. {\bf 266} (1984), pp. 287-295.
    \bibitem{man3}D. Manderscheid, {\it Supercuspidal duality for the two-fold cover of $\SL_2$ and the split $\Or_3$}, Amer. J. Math.,{\bf 107} no. 6, (1985), pp. 1305-1324.
    \bibitem{rao}R. Ranga Rao, {\it On some explicitformulas in the theory of Weil representation}, Pacific Journal of Math. {\bf 157} No.2, (1993), pp. 335-371.
    \bibitem{rs-GSp4}B. Roberts and R. Schmidt, {\it Local newforms for $\GSp(4)$}, Lecture Notes in Mathematics, 1918. Springer, Berlin, (2007).
    \bibitem{rs-Mp2}B. Roberts and R. Schmidt, {\it On the number of local newforms in a metaplectic representation}, Arithmetic geometry and automorphic forms, pp. 505-530, Adv. Lect. Math. (ALM), 19, Int. Press, Somerville, MA, (2011).
    \bibitem{shi}G. Shimura, {\it On modular forms of half integral weight}, Ann. of Math. (2) {\bf 97} (1973), pp. 440-481.
    \bibitem{szp}D. Szpruch, {\it Computation of the local coefficients for principal series representations of the metaplectic double cover of $\SL_2(\mathbb{F})$}, J. Number Theory, {\bf 129} (2009) no. 9, pp. 2180-2213.
    \bibitem{tan}S. Tanaka, {\it  Construction and classification or irreducible representations of special linear group of the second order over a finite field}, Osaka J. Math. {\bf 4} (1967), pp. 65-84.
    \bibitem{tsai}P.-Y. Tsai, {\it On newforms for split special odd orthogonal groups}, PhD thesis, Harvard, (2013).
    \bibitem{tun}J.-B. Tunnell, {\it On the local Langlands conjecture for $GL(2)$}, Invent. Math., {\bf 46} (1978), no. 2, pp. 179-200.
    \bibitem{uy}M. Ueda and S. Yamana,{\it On newforms for Kohnen plus spaces}, Math. Z. {\bf 264} (2010), pp. 1-13.
    \bibitem{wal1}J.-L. Waldspurger, {\it Correspondance de Shimura}, J. Math. Pures et Appl. 59 (1980), pp. 1-133.
    \bibitem{wal2}J.-L. Waldspurger, {\it Correspondance de Shimura et quaternions}, Forum Math. 3 (1991), pp. 219-307.
    \bibitem{yam}Y. Yamamoto, {\it On Mackey decomposition for locally profinite groups}, arXiv: 2203. 14262, (2022).
\end{thebibliography}
\end{document}